\documentclass[12pt]{article} 

\usepackage[english]{babel}
\usepackage{amsfonts}
\usepackage{amstext}
\usepackage{graphicx}
\usepackage{amsmath}
\usepackage{amssymb}
\usepackage{euscript}
\usepackage{amsthm}
\usepackage{enumerate}

\renewcommand{\Re}{\mathop{\rm Re}\nolimits}
\renewcommand{\Im}{\mathop{\rm Im}\nolimits}

\numberwithin{equation}{section}

\newtheorem{theorem}{Theorem}
\newtheorem{proposition}{Proposition}
\newtheorem{remark}{Remark}

\begin{document}

\title{Long-time asymptotics for the integrable nonlocal nonlinear Schr\"odinger equation with step-like initial data}
\author{Ya. Rybalko$^{\dag,\ddag}$ and D. Shepelsky$^{\dag,\ddag}$\\
	\small\em {}$^\dag$ B.Verkin Institute for Low Temperature Physics and Engineering\\
	\small\em {} of the National Academy of Sciences of Ukraine\\
	\small\em {}$^\ddag$ V.Karazin Kharkiv National University}

\maketitle

\begin{abstract}
We study the Cauchy problem for 
the integrable nonlocal nonlinear Schr\"odinger (NNLS) equation
\[
iq_{t}(x,t)+q_{xx}(x,t)+2 q^{2}(x,t)\bar{q}(-x,t)=0
\]
with a step-like initial data: $q(x,0)=q_0(x)$, where 
$q_0(x)=o(1)$ as $x\to-\infty$ and  $q_0(x)=A+o(1)$ as $x\to\infty$, with
an arbitrary positive constant  $A>0$. The main aim  is to study
the long-time behavior of the solution of this  problem.
We show that the asymptotics has qualitatively different form in the quarter-planes
of the half-plane $-\infty<x<\infty$, $t>0$:  (i) for $x<0$, the solution approaches 
 a slowly decaying, modulated wave of the Zakharov-Manakov type; (ii) for $x>0$,  the solution approaches the ``modulated constant''. The main tool is the representation of the solution 
of the Cauchy problem in terms of the solution of an associated matrix Riemann-Hilbert 
(RH) problem and the consequent  asymptotic analysis of this RH problem.
\end{abstract}



\section{Introduction}
We consider the following initial value  problem for the focusing nonlocal nonlinear Schr\"odinger (NNLS) equation with a step-like initial data:
\begin{subequations}
\label{1}
\begin{align}
\label{1-a}
& iq_{t}(x,t)+q_{xx}(x,t)+2q^{2}(x,t)\bar{q}(-x,t)=0, & & x\in\mathbb{R},\,t>0,  \\
\label{1-b}
& q(x,0)=q_0(x), & &  x\in\mathbb{R}, 
\end{align}
where 
\begin{equation}
\label{1-c}
q_0(x) \to 0\ \text{ as}\  x\to -\infty\ \ \text{and}\ \   q_0(x) \to A \text{ as}\ x\to \infty
\end{equation}
sufficiently fast, with some $A>0$. 
\end{subequations}
Throughout the paper, $\bar{q}$ denotes the complex conjugate of $q$.

The nonlocal nonlinear Schr\"odinger equation in the form (\ref{1-a}) 
was introduced by M. Ablowitz and Z. Musslimani in \cite{AMP}. Although this equation is just a reduction of a member of the AKNS hierarchy \cite{AKNS}, namely, of the coupled Schr\"odinger equations
\begin{subequations}
\label{cs}
\begin{align}
\label{cs-a}
iq_t+q_{xx}+2q^2r=0,\\
\label{cs-b}
-ir_t+r_{xx}+2r^2q=0,
\end{align}
\end{subequations}
corresponding to $r(x,t)=\bar{q}(-x,t)$, the NNLS equation has recently attracted much attention  because of its distinctive physical and mathematical properties.
Indeed, this equation is invariant under the joint transformations $x\to-x$, $t\to-t$, and complex conjugation, i.e. it is parity-time (PT) symmetric and, therefore, is related to a cutting edge research area of modern physics \cite{BB, KYZ}. Particularly, due to the gauge-equivalence of the NNLS to the unconventional system of coupled Landau-Lifshitz (CLL) equations, this equation can find applications in the physics of nanomagnetic artificial materials \cite{GA}.

Because of these features of the NNLS equation and the potential applications,  other symmetry reductions of the AKNS and other hierarchies, which lead to other types of nonlocality,  began to attract considerable attention. 
Typical examples are the reverse space-time nonlocal NLS equation and the reverse time nonlocal NLS equation, the complex/real space-time Sine-Gordon equation, the complex/real reverse space-time mKdV equation \cite{AFLM17, AMN, AM17}, the nonlocal derivative NLS equation \cite{Z}, and the multidimensional nonlocal Davey-Stewartson equation \cite{AM17, F16}.

In \cite{AMN} the authors presented the Inverse Scattering Transform (IST) method to the study of the Cauchy problem for  equation (\ref{1-a}), based on a variant of the Riemann-Hilbert approach, in the case of decaying initial data and obtained the one- and two-soliton solutions. In \cite{AMFL} and \cite{Y}, a general decaying N-soliton solution of  (\ref{1-a}) were found using the  Hirota's direct method and the Riemann-Hilbert approach respectively (see also \cite{Ybook}, where the N-soliton solution of the general coupled Schr\"odinger equations 
(\ref{cs}) is presented by the Riemann-Hilbert approach). The one-, two- and three-soliton solutions are obtained via the Hirota's direct method in 
\cite{GP} whereas in \cite{CZ},  the decaying one-soliton solution is obtained in terms of a double Wronskian. 
The soliton solutions of the focusing NNLS equation (\ref{1-a}) have some specific features: particularly, they can blow up at a finite time \cite{AMFL, AMN}, and (\ref{1-a}) can simultaneously support both bright and dark soliton solutions \cite{SMMC}.

The initial value problem for  (\ref{1-a}) with the following nonzero boundary conditions: 
\begin{equation}\label{Abc}
q(x,t)\to q_{\pm}(t)=q_0e^{i(\alpha t+\theta_{\pm})},\qquad x\to\pm\infty,
\end{equation}
where $q_0>0$, $\alpha\in\mathbb{R}$, $0\leq\theta_{\pm}<2\pi$, is considered in \cite{ALM16}, 
where the IST method is developed and the soliton solutions are constructed for certain values 
of the parameters $\theta_{\pm}$ (see also \cite{AMFL}, where the general 
$N$-soliton solutions are presented).

In the present paper we assume that the solution $q(x,t)$ of  problem (\ref{1-a}-\ref{1-b}) satisfies the following boundary conditions  for all $t>0$:
\begin{subequations}
\label{2.0}
\begin{align}
\label{2-a}
& q(x,t)=o(1), & x\rightarrow-\infty,\\
\label{2-b}
& q(x,t)=A+o(1), & x\rightarrow+\infty
\end{align}
\end{subequations}
(in what follows we will make the sense of $o(1)$ more precise).
This choice of initial data and boundary values is inspired by the shock problems for the classical (local) NLS equation 
\begin{equation}
\label{NLS}
iq_{t}(x,t)+q_{xx}(x,t)+2|q(x,t)|^2q(x,t)=0,
\end{equation}
which is another (local) reduction of system (\ref{cs}), with $r(x,t)=\bar q(x,t)$.
Such problems 
have been considered since 1980s \cite{B93,  BK, BKS, AU, KK}.
Particularly, in \cite{BKS} the authors study the Cauchy problem for the NLS equation with the following initial condition:
\begin{equation}
\label{2}
q_0(x)=
\begin{cases}
0, & x\leq 0,\\
Ae^{-2iBx}, & x>0,\\
\end{cases}
\end{equation}
assuming that the solution satisfies the boundary conditions
\begin{subequations}
\label{3}
\begin{align}
\label{3-a}
& q(x,t)=o(1),& x\rightarrow-\infty,\\
\label{3-b}
& q(x,t)=q^{p}(x,t)+o(1),& x\rightarrow+\infty,
\end{align}
\end{subequations}
where $q^{p}(x,t)=Ae^{-2iBx+2i\omega t}$ with $\omega=A^2-2B^2$ is a plane wave solution of 
the NLS equation (\ref{NLS}). 
Notice that for the classical NLS, 
the both limiting functions in (\ref{3}), i.e., $q_-(x,t)\equiv 0 $
and $q_+(x,t)=q^{p}(x,t)$ are solutions of (\ref{NLS})
whereas in the case of the NNLS equation,  $q_-(x,t)\equiv 0$ is a solution,
but $q_+(x,t)\equiv A$ is not.
With this respect,  the non-zero boundary conditions (\ref{2.0}), being  the simplest shock-type boundary conditions for the NNLS equation (\ref{1-a}),
differ from  
 those  used for the local NLS equation.

The present paper aims at (i) the development of the  Riemann-Hilbert approach 
to the initial value problem (\ref{1}) with the boundary conditions (\ref{2.0}) and (ii)
the long-time asymptotic analysis of solutions to this problem using the nonlinear steepest-decent method \cite{DZ}.
The nonlinear steepest-decent method was inspired by earlier works by Manakov \cite{M} and Its \cite{I1} (see \cite{DIZ} for a comprehensive historical review) and has been put into a rigorous shape by Deift and Zhou in \cite{DZ}, with further extensions in  \cite{DVZ94, DVZ97}. The nonlinear steepest-decent method is known to be extremely efficient for the asymptotic analysis of a wide variety of initial and initial boundary value problems for integrable systems, particularly, it has been successfully applied to many initial value problems with step-like initial data, see, e.g.,  \cite{BK14, BM17, BKS, BV07, D14, KM, XFC}.

The  paper is organized as follows.  In Section 2 we present the formalism of the IST method in the form of a multiplicative RH problem suitable for the asymptotic (as $t\to\infty$) analysis. Here we emphasize specific features
of the implementation of the Riemann-Hilbert problem formalism
in our case, one of them being a singularity, of 
particular (different for different cases of initial data) type,
at the jump contour of  the RH problem.  The long-time asymptotic  analysis
of the main RH problem (and, consequently, of the solution of the Cauchy problem for 
the NNLS equation) is then presented in Section 3,
where the main result of the paper (Theorem \ref{th1}) is formulated.
Two main peculiar aspects of our asymptotic results are (i) the dependence
of the power-type decay parts of the asymptotics on the direction $x/t=const$ 
(recall that in the case of the local NLS equation (as well as for other integrable 
equations like the (local)  Korteweg-de Vries equation, 
the modified Korteweg-de Vries equation, etc.), 
the corresponding power decay is $t^{-1/2}$ independently of the direction);
(ii) the absence of a sector in the $(x,t)$ plane, with straight boundaries
$x/t=c_1$ and $x/t=c_2$, where the main term of the asymptotics 
is described in terms of modulated elliptic functions (which, again,
is typical for local integrable  nonlinear equations,
with step-like initial data, including the local NLS equation \cite{BM17, BKS}).

\section{Inverse scattering transform and the Riemann-Hilbert problem}\label{ist}

\subsection{Eigenfunctions}
Recall that the focusing NNLS equation (\ref{1-a}) is a compatibility condition of the following two linear equations (Lax pair) \cite{AKNS, ALM16}
\begin{equation}
\label{LP}
\left\{
\begin{array}{lcl}
\Phi_{x}+ik\sigma_{3}\Phi=U(x,t)\Phi\\
\Phi_{t}+2ik^{2}\sigma_{3}\Phi=V(x,t,k)\Phi\\
\end{array}
\right.
\end{equation}
where $\sigma_3=\left(\begin{smallmatrix} 1& 0\\ 0 & -1\end{smallmatrix}\right)$, $\Phi(x,t,k)$ is a $2\times2$ matrix-valued function, $k\in\mathbb{C}$ is an auxiliary (spectral) parameter, and the matrix coefficients $U(x,t)$ and $V(x,t,k)$ are given in terms of  $q(x,t)$:	
\begin{equation}
U(x,t)=\begin{pmatrix}
0& q(x,t)\\
-\bar{q}(-x,t)& 0\\
\end{pmatrix},\qquad 
V=\begin{pmatrix}
V_{11}& V_{12}\\
V_{21}& V_{22}\\
\end{pmatrix},
\end{equation}
 where $V_{11}=-V_{22}=iq(x,t)\bar{q}(-x,t)$, $V_{12}=2kq(x,t)+iq_{x}(x,t)$, and
$V_{21}=-2k\bar{q}(-x,t)+i(\bar{q}(-x,t))_{x}$. 

Introduce the  notations
\begin{equation}
U_+=
\begin{pmatrix}
0 & A\\
0 & 0
\end{pmatrix},\,
U_-=
\begin{pmatrix}
0 & 0\\
-A & 0
\end{pmatrix},\,
V_+=
\begin{pmatrix}
0 & 2kA\\
0 & 0
\end{pmatrix},\,
V_-=
\begin{pmatrix}
0 & 0\\
-2kA & 0
\end{pmatrix}.
\end{equation}
Then, assuming that there exists $q(x,t)$ satisfying (\ref{1}) and (\ref{2.0}), it follows that 
\begin{equation}
U(x,t)\rightarrow U_{\pm}\ \mbox{and}\ 
V(x,t,k)\rightarrow V_{\pm}(k)\quad \mbox{as}\  x\rightarrow\pm\infty.
\end{equation}
It is easy to see that
 the systems 
 $$
\begin{cases}
\Phi_x+ik\sigma_3\Phi=U_{+}\Phi\\
\Phi_t+2ik^2\sigma_3\Phi=V_{+}(k)\Phi
\end{cases}
$$
and 
$$
\begin{cases}
\Phi_x+ik\sigma_3\Phi=U_{-}\Phi\\
\Phi_t+2ik^2\sigma_3\Phi=V_{-}(k)\Phi
\end{cases}
$$
are compatible (cf. (2.1)). 
Particularly, they are satisfied by $\Phi_{\pm}(x,t,k)$ defined as follows:
\begin{equation}
\label{5}
\Phi_{\pm}(x,t,k)=N_{\pm}(k)e^{-(ikx+2ik^2t)\sigma_3},
\end{equation}
where 
$N_+(k)=
\begin{pmatrix}
1 & \frac{A}{2ik}\\
0 & 1
\end{pmatrix}$ and  
$N_-(k)=
\begin{pmatrix}
1 & 0\\
\frac{A}{2ik} & 1
\end{pmatrix}$.
Notice that $\Phi_{\pm}$ are chosen in such a way that $\det \Phi_{\pm}\equiv 1$,
which is convenient for the analysis that follows, particularly, when considering the uniqueness issue
in the Riemann-Hilbert problem. On the other hand, 
the singularities of  $N_{\pm}(k)$ at  $k=0$ 
will significantly affect  this analysis.
  Namely, the solution of the basic RH problem has a singularity as $k\to 0$,
i.e. at a point on the contour of the RH problem
(see (\ref{z}) and (\ref{nz}) below).

Now define  the $2\times2$-valued functions $\Psi_j(x,t,k)$, $j=1,2$, $-\infty<x<\infty$, $0\leq t<\infty$ as the solutions of the Volterra integral equations:
\begin{subequations}\label{6}
\begin{align}
\label{6-1}
& \Psi_1(x,t,k)=N_-(k)+\int^x_{-\infty}G_-(x,y,t,k)\left(U(y,t)-U_-\right)\Psi_1(y,t,k)e^{ik(x-y)\sigma_3}\,dy,\\
\label{7}
& \Psi_2(x,t,k)=N_+(k)+\int^x_{\infty}G_+(x,y,t,k)\left(U(y,t)-U_+\right)\Psi_2(y,t,k)e^{ik(x-y)\sigma_3}\,dy,
\end{align}
where $G_{\pm}(x,y,t,k)=\Phi_{\pm}(x,t,k)[\Phi_{\pm}(y,t,k)]^{-1}$.
\end{subequations}
The functions $\Psi_{j}(x,t,k)$, $j=1,2$ are the main ingredients of the basic RH problem (see (\ref{DM}) below).
The main properties of the matrices $\Psi_j(x,t,k)$ 
(following from the integral equations (\ref{6}))
are summarized in  Proposition \ref{prop1}, where we denote by $\Psi_j^{(i)}(x,t,k)$ the i-th column of $\Psi_j(x,t,k)$,  
$\mathbb{C}^{\pm}=\left\{k\in\mathbb{C}\,|\pm\Im k>0\right\}$, and 
$\overline{\mathbb{C}^{\pm}}=\left\{k\in\mathbb{C}\,|\pm\Im k\ge 0\right\}$.
\begin{proposition}
\label{prop1}
The matrices $\Psi_1(x,t,k)$ and $\Psi_2(x,t,k)$ have the following properties:
\begin{enumerate}[(i)]
\item The columns $\Psi_1^{(1)}(x,t,k)$ and $\Psi_2^{(2)}(x,t,k)$ are well-defined and analytic in $k\in\mathbb{C}^+$ and continuous in $\overline{\mathbb{C}^+}\setminus\{0\}$; moreover,
\[
\Psi_1^{(1)}(x,t,k)=
\begin{pmatrix}
1\\
0\end{pmatrix}
+O(k^{-1}) \ \text{and}\ \ \Psi_2^{(2)}(x,t,k)=
\begin{pmatrix}
0\\
1\end{pmatrix}
+O(k^{-1}) \quad \text{as}\ 
 k\rightarrow\infty, \quad  k\in\mathbb{C}^+.
\]
\item The columns $\Psi_1^{(2)}(x,t,k)$ and $\Psi_2^{(1)}(x,t,k)$ are well-defined and analytic in $k\in\mathbb{C}^-$ and continuous in $\overline{\mathbb{C}^-}$; moreover,
\[
\Psi_1^{(2)}(x,t,k)=
\begin{pmatrix}
0\\
1\end{pmatrix}
+O(k^{-1}) \ \text{and}\ \ 
\Psi_2^{(1)}(x,t,k)=
\begin{pmatrix}
1\\
0\end{pmatrix}
+O(k^{-1})\quad \text{as}\  k\rightarrow\infty,\quad k\in\mathbb{C}^-.
\]

\item The functions $\Phi_j(x,t,k)$, $j=1,2$ defined by
\begin{equation}
\label{7.5}
\Phi_j(x,t,k)=\Psi_j(x,t,k)e^{-(ikx+2ik^2t)\sigma_3},\qquad k\in\mathbb{R}\setminus \{0\},\qquad j=1,2,
\end{equation}
are the  (Jost) solutions of the Lax pair equations (\ref{LP}) satisfying 
the boundary conditions  
\begin{subequations} 
\begin{align}
&\Phi_1(x,t,k)\rightarrow\Phi_-(x,t,k),\qquad x\rightarrow-\infty,\\
&\Phi_2(x,t,k)\rightarrow\Phi_+(x,t,k),\qquad x\rightarrow\infty.
\end{align}
\end{subequations}

\item $\det\Psi_j(x,t,k)\equiv 1$,$\qquad$ $x\in\mathbb{R}$, $t\geq0$, $k\in\mathbb{R}$, $\qquad j=1,2$.

\item The  following symmetry relation holds:
\begin{equation}
\label{7.9}
\Lambda\overline{\Psi_1(-x,t,-k)}\Lambda^{-1}=\Psi_2(x,t,k),\,\,k\in\mathbb{R}\setminus\{0\},
\end{equation}
where $\Lambda=\bigl(\begin{smallmatrix}0& 1\\1 & 0\end{smallmatrix}\bigl)$.

\item As $k\to 0$,
\begin{subequations}
\label{k0}
\begin{align}
\label{k0-a}
&\Psi_1^{(1)}(x,t,k)=\frac{1}{k}
\begin{pmatrix}v_1(x,t)\\v_2(x,t)\end{pmatrix}+O(1),
& & \Psi_1^{(2)}(x,t,k)=\frac{2i}{A}\begin{pmatrix}v_1(x,t)\\ v_2(x,t)\end{pmatrix}+O(k),\\
\label{k0-b}
&\Psi_2^{(1)}(x,t,k)=-\frac{2i}{A}
\begin{pmatrix}
\overline{v_2}(-x,t)\\ \overline{v_1}(-x,t)
\end{pmatrix}+O(k),
& & \Psi_2^{(2)}(x,t,k)=
-\frac{1}{k}
\begin{pmatrix}\overline{v_2}(-x,t)\\\overline{v_1}(-x,t)
\end{pmatrix}+O(1),
\end{align}
\end{subequations}
where $v_j(x,t)$, j=1,2 solve the following system of  Volterra integral equations:
\begin{equation}
\label{vv}
\begin{cases}
v_1(x,t)=\int_{-\infty}^{x}q(y,t)v_2(y,t)\,dy,\\
v_2(x,t)=-i\frac{A}{2}-\int_{-\infty}^{x}
\overline{q(-y,t)}v_1(y,t)\,dy.
\end{cases}
\end{equation}
\end{enumerate}
\end{proposition}

\begin{proof}
Properties (i)-(iii) follows directly from the representation of  $\Psi_j$ 
in terms of the  Neumann series associated with equations (\ref{6}).
The Neumann series  converge provided 
$\int_{-\infty}^0 |q(x,t)|dx <\infty$ and 
$\int^\infty_0 |q(x,t)-A|dx <\infty$ for all $t\ge 0$ (cf. (\ref{2.0})).
Item (iv) follows from the fact that $U$ and $V$ in (\ref{LP}) are traceless.
Item (v) follows from the corresponding symmetry 
$\Lambda\overline{U(-x,t)}\Lambda^{-1}=U(x,t)$.

Now let us  discuss Item (vi). From (\ref{6}) and the structure of singularity
of $N_\pm$ at $k=0$ it follows that, as $k\to 0$, 
\begin{subequations}
\label{vw}
\begin{align}
&\Psi_1^{(1)}(x,t,k)=\frac{1}{k}\begin{pmatrix}v_1(x,t)\\ v_2(x,t)\end{pmatrix}+O(1),
& & \Psi_1^{(2)}(x,t,k)=\begin{pmatrix}\tilde{v}_1(x,t)\\ \tilde{v}_2(x,t)\end{pmatrix}+O(k),\\
&\Psi_2^{(1)}(x,t,k)=\begin{pmatrix}\tilde{w}_1(x,t)\\ \tilde{w}_2(x,t)\end{pmatrix}+O(k),
& & \Psi_2^{(2)}(x,t,k)=\frac{1}{k}\begin{pmatrix}w_1(x,t)\\ w_2(x,t)\end{pmatrix}+O(1)
\end{align}
\end{subequations}
with some $v_j$, $\tilde v_j$, $w_j$ and $\tilde w_j$ ($j=1,2$).
Then, the symmetry relation (\ref{7.9}) implies that
\begin{equation}
\begin{pmatrix}
w_1(x,t)\\w_2(x,t)
\end{pmatrix}=
\begin{pmatrix}
-\overline{v_2}(-x,t)\\
-\overline{v_1}(-x,t)
\end{pmatrix}\quad\text{and}\quad
\begin{pmatrix}\tilde{w}_1(x,t)\\ \tilde{w}_2(x,t)
\end{pmatrix}=
\begin{pmatrix}
\overline{\tilde{v}_2}(-x,t)\\
\overline{\tilde{v}_1}(-x,t)
\end{pmatrix}.
\end{equation}
Further, substituting (\ref{vw})  into 
(\ref{6})  we conclude that
$v_j(x,t)$, $j=1,2$ satisfy  (\ref{vv}) whereas 
$\tilde{v}_j(x,t)$, $j=1,2$ solve the following system of equations
\begin{equation}\label{2.14}
\begin{cases}
\tilde{v}_1(x,t)=\int_{-\infty}^{x}q(y,t)\tilde{v}_2(y,t)\,dy,\\
\tilde{v}_2(x,t)=1-\int_{-\infty}^{x}\overline{q(-y,t)}\tilde{v}_1(y,t)
\,dy.
\end{cases}
\end{equation}
Comparing (\ref{2.14}) with (\ref{vv}), it follows that 
\begin{equation}
\begin{pmatrix}
\tilde{v}_1(x,t)\\\tilde{v}_2(x,t)
\end{pmatrix}=
\frac{2i}{A}\begin{pmatrix}
v_1(x,t)\\v_2(x,t)
\end{pmatrix}
\end{equation}
and thus (\ref{k0}) can be characterized in terms of two functions only, $v_1(x,t)$ and 
$v_2(x,t)$.
\end{proof}
\subsection{Scattering data}
Since  $\Phi_1(x,t,k)$ and $\Phi_2(x,t,k)$ 
are both  well-defined for $k\in\mathbb{R}\setminus\{0\}$ and
satisfy the both equations in the Lax pair 
(\ref{LP}),
it follows that  
\begin{equation}
\label{8}
\Phi_1(x,t,k)=\Phi_2(x,t,k)S(k),\,\,k\in\mathbb{R}\setminus\{0\},
\end{equation}
or, in terms of $\Psi_j$,
\begin{equation}
\label{9}
\Psi_1(x,t,k)=\Psi_2(x,t,k)e^{-(ikx+2ik^2t)\sigma_3}S(k)e^{(ikx+2ik^2t)\sigma_3},\qquad 
k\in\mathbb{R}\setminus\{0\},
\end{equation}
where $S(k)$ is called the scattering matrix. 

The symmetry relation (\ref{7.9})
implies that the same relation holds for the Jost solutions $\Phi_1(x,t,k)$ and $\Phi_2(x,t,k)$:
\begin{equation}\label{phi-sym}
\Lambda\overline{\Phi_1(-x,t,-\bar{k})}\Lambda^{-1}=\Phi_2(x,t,k),\qquad k\in\mathbb{R}\setminus\{0\}.
\end{equation}
In turn, this implies that the scattering matrix $S(k)$ can be written as follows (cf. \cite{AMN, RS})
\begin{equation}
S(k)=
\begin{pmatrix}
a_1(k)& -\overline{b(-k)}\\
b(k)& a_2(k)
\end{pmatrix},\qquad k\in\mathbb{R}\setminus\{0\},
\end{equation}
with some $b(k)$, $a_1(k)$, and $a_2(k)$; moreover,  $a_1(k)$ and $a_2(k)$ are well defined in  $\overline{{\mathbb C}^+}\setminus\{0\}$ and $\overline{{\mathbb C}^-}$ respectively, where they satisfy 
the symmetry relations
\begin{equation}
\label{a-sym}
\overline{a_1(-\bar{k})} = a_1(k), \qquad \overline{a_2(-\bar{k})} = a_2(k).
\end{equation}

The scattering matrix $S(k)$ is uniquely determined by the initial data $q_0(x)$. Indeed, introducing the notations $\psi_1(x,k)=(\Psi_1)_{11}(x,0,k)$, $\psi_2(x,k)=(\Psi_1)_{12}(x,0,k)$, $\psi_3(x,k)=(\Psi_1)_{21}(x,0,k)$ and $\psi_4(x,k)=(\Psi_1)_{22}(x,0,k)$,
equations (\ref{6-1}) reduce to 
the systems of  Volterra integral equations for $\psi_1$ and $\psi_3$:
\begin{gather}
\label{8.1}
\left\{
\begin{array}{lcl}
\psi_1(x,k)=1+\int_{-\infty}^{x}q_{0}(y)\psi_3(y,k)\,dy,\\
\psi_3(x,k)=\frac{A}{2ik}+\int_{-\infty}^{x}e^{2ik(x-y)}\left(A-\overline{q_{0}(-y)}\right)\psi_1(y,k)\,dy
+\frac{A}{2ik}\int_{-\infty}^{x}q_0(y)\left(1-e^{2ik(x-y)}\right)\psi_3(y,k)\,dy\\
\end{array}
\right.
\end{gather}
and for $\psi_2$ and $\psi_4$:
\begin{gather}
\label{8.3}
\left\{
\begin{array}{lcl}
\psi_2(x,k)=\int_{-\infty}^{x}e^{-2ik(x-y)}q_{0}(y)\psi_4(y,k)\,dy,\\
\psi_4(x,k)=1+\int_{-\infty}^{x}\left(A-\overline{q_{0}(-y)}\right)\psi_2(y,k)\,dy
+\frac{A}{2ik}\int_{-\infty}^{x}q_0(y)\left(e^{-2ik(x-y)}-1\right)\psi_4(y,k)\,dy.\\
\end{array}
\right.
\end{gather}
Then the entries $a_1$, $a_2$ and $b$ of the scattering matrix
can be determined as follows:
\begin{equation}
\label{8.2}
a_{1}(k)=\lim\limits_{x\rightarrow\infty}\left(\psi_1(x,k)-\frac{A}{2ik}\psi_3(x,k)\right),
\quad b(k)=\lim\limits_{x\rightarrow\infty}e^{-2ikx}\psi_3(x,k),
\end{equation}
and 
\begin{equation}
\label{8.4}
a_{2}(k)=\lim\limits_{x\rightarrow\infty}\psi_4(x,k).
\end{equation}

Alternatively, they can be 
written it terms of  the determinant relations:
\begin{subequations}\label{sd}
\begin{align}
a_1(k)&=\det\left(\Psi_1^{(1)}(0,0,k),\Psi_2^{(2)}(0,0,k)\right),
\quad k\in\overline{\mathbb{C}^{+}}\setminus\{0\},\\
a_2(k)&=\det\left(\Psi_2^{(1)}(0,0,k),\Psi_1^{(2)}(0,0,k)\right),
\quad k\in\overline{\mathbb{C}^{-}},\\
b(k)&=\det\left(\Psi_2^{(1)}(0,0,k),\Psi_1^{(1)}(0,0,k)\right),
\quad k\in\mathbb{R}.
\end{align}
\end{subequations}

The properties of the spectral functions, which follow from Proposition \ref{prop1}, are summarized in 
\begin{proposition}\label{properties}
The spectral functions $a_j(k)$, j=1,2, and $b(k)$ have the following properties
\begin{enumerate}
	\item 
	$a_{1}(k)$ is analytic in  $k\in\mathbb{C}^{+}$
and continuous in 
$\overline{\mathbb{C}^{+}}\setminus\{0\}$;
 $a_{2}(k)$ is analytic in $k\in\mathbb{C}^{-}$
and continuous in 
$\overline{\mathbb{C}^{-}}$. 
\item
 $a_{j}(k)=1+{O}\left(\frac{1}{k}\right)$, $j=1,2$ 
 as $k\rightarrow\infty$, $k\in\overline{\mathbb{C}^{(-1)^{j+1}}}$
 and 
$b(k)={O}\left(\frac{1}{k}\right)$ as  $k\rightarrow\infty$, 
$k\in\mathbb{R}$.
\item
$\overline{a_{1}(-\bar{k})}=a_1(k)$,  
$k\in\overline{\mathbb{C}^{+}}\setminus\{0\}$; \qquad
$\overline{a_{2}(-\bar{k})}=a_2(k)$,  
$k\in\overline{\mathbb{C}^{-}}$.
\item
 $a_{1}(k)a_{2}(k)+b(k)\overline{b(-k)}=1$, 
 $k\in{\mathbb R}\setminus\{0\}$ (follows from $\det S(k)=1$).
\item 
$a_1(k)=\frac{A^2a_2(0)}{4k^2}+O(\frac{1}{k})$ as $k\to0$, $k\in \overline{\mathbb{C}^{+}}$ and $b(k)=\frac{Aa_2(0)}{2ik}+O(1)$ as $k\to0$, $k\in\mathbb{R}$.
\end{enumerate}
\end{proposition}

\begin{remark}\label{prop2-proof}
Concerning  Item 5 of Proposition \ref{properties}, we notice 
that substituting (\ref{k0}) into (\ref{sd}) yields, as $ k\to 0$,
\begin{subequations}
\label{sz}
\begin{align}
a_1(k)&=\frac{1}{k^2}(|v_2(0,0)|^2-|v_1(0,0)|^2)+O\left(\frac{1}{k}\right),\\
\label{sz2}
a_2(k)&=\frac{4}{A^2}(|v_2(0,0)|^2-|v_1(0,0)|^2)+O(k),\\
b(k)&=-\frac{2i}{kA}(|v_2(0,0)|^2-|v_1(0,0)|^2)+O(1),
\end{align}
\end{subequations}
from which Item 5 follows. Notice that in (\ref{sd})  
one can use any $(x,t)$ instead of $(0,0)$ as arguments in the right-hand sides,
which implies that $|v_2(0,0)|^2-|v_1(0,0)|^2$
in the r.h.s. of (\ref{sz}) can be replaced by 
$v_2(x,t)\bar v_2(-x,t)-v_1(x,t)\bar v_1(-x,t)$, the latter being a conserved quantity (independent of $x$ and $t$).
\end{remark}

\begin{remark}\label{r1}
In the case of the pure-step initial data, i.e., when
\begin{equation}
\label{step-ini}
q_0(x) = q_{0A}(x):=\begin{cases} 0, & x<0, \\
A, & x>0,
\end{cases}
\end{equation}
 the scattering matrix $S(k)$ is as follows:
\begin{equation}
\label{8.5}
S(k)=[\Phi_2(0,0,k)]^{-1}\Phi_1(0,0,k)=N_+^{-1}(k)N_-(k)=
\begin{pmatrix}
1+\frac{A^2}{4k^2}& -\frac{A}{2ik}\\
\frac{A}{2ik} & 1
\end{pmatrix}.
\end{equation}
Particularly, in this case $a_1(k)$  has a single, simple zero (at $k=i\frac{A}{2}$) in the upper half-plane  
  whereas $a_2(k)$ has no zeros in the lower half-plane.
\end{remark}

\subsection{The basic Riemann-Hilbert problem}

The Riemann--Hilbert formalism of the IST method is based on constructing (using the Jost soultions)
 a piece-wise meromorphic,
$2\times 2$-valued function in the $k$-complex plane, whose ``lack of analyticity'',
i.e., the jump across a contour and, if appropriate, some conditions at the singularity points,
can be fully characterized in terms of the spectral data (spectral functions and a discrete 
set of data related to the poles) uniquely determined by the initial data.

Define the $2\times 2$-valued function $M(x,t,k)$, piece-wise meromorphic relative to 
$\mathbb R$, as follows:

\begin{equation}
\label{DM}
M(x,t,k)=
\left\{
\begin{array}{lcl}
\left(\frac{\Psi_1^{(1)}(x,t,k)}{a_{1}(k)},\Psi_2^{(2)}(x,t,k)\right),\quad k\in\mathbb{C}^+,\\
\left(\Psi_2^{(1)}(x,t,k),\frac{\Psi_1^{(2)}(x,t,k)}{a_{2}(k)}\right),\quad k\in\mathbb{C}^-.\\
\end{array}
\right.
\end{equation}
Then the scattering relation (\ref{9}) implies that the boundary values 
$M_\pm(x,t,k) = \underset{k'\to k, k'\in {\mathbb C}^\pm}{\lim} M(x,t,k')$, $k\in\mathbb R$
 satisfy the  multiplicative jump condition
\begin{equation}\label{jr}
M_+(x,t,k)=M_-(x,t,k)J(x,t,k),\qquad k\in\mathbb{R}\setminus\{0\},
\end{equation}
where
\begin{equation}\label{jump}
J(x,t,k)=
\begin{pmatrix}
1+r_{1}(k)r_{2}(k)& r_{2}(k)e^{-2ikx-4ik^2t}\\
r_1(k)e^{2ikx+4ik^2t}& 1
\end{pmatrix}
\end{equation}
with the reflection coefficients defined by
\begin{equation}\label{r12}
r_1(k):=\frac{b(k)}{a_1(k)},\quad r_2(k):=\frac{\overline{b(-k)}}{a_2(k)}.
\end{equation}
Moreover, $M$ satisfies  the normalization condition
\begin{equation}
M(x,t,k)\to I,\qquad k\to\infty,
\end{equation}
where  $I$ is the $2\times2$ identity matrix.

Observe that the symmetry conditions 3 in Proposition \ref{properties} imply that
\begin{equation}\label{r-sym}
r_1(-k) r_2(-k)  = \overline{r_1(k)}\;  \overline{r_2(k)}, \quad k\in {\mathbb R\setminus\{0\}}.
\end{equation}
By the determinant property 4,  we also have
\begin{equation}\label{r-a}
1+r_1(k) r_2(k)=\frac{1}{a_1(k)a_2(k)},\qquad k\in {\mathbb R\setminus\{0\}}.
\end{equation}

Now notice that in view of (\ref{sz}), the behavior of $M$ as $k\to 0$ 
is 
 qualitatively different in the cases 
 $a_2(0)\ne 0$ and $a_2(0)= 0$. The former case
contains the case of ``pure-step initial data'', 
see Remark \ref{r1}, where $a_1(k)$ has (in ${\mathbb C}^+$) a single, simple zero located on the imaginary axis, and $a_2(k)$ has no zeros in  
${\mathbb C}^-$. Since small (in the $L^1$ norm) perturbations of the pure-step initial data
preserve these properties, we will concentrate, in the present paper, on the following two cases:

\begin{description}
\item [Case I:] 
The spectral function $a_1(k)$ has one (pure imaginary) simple zero in $\overline{\mathbb{C}^+}$, say $k=ik_1$, $k_1>0$, 
and $a_2(k)$ has no zeros in $\overline{\mathbb{C}^-}$.
\item [Case II:] 
The spectral function $a_1(k)$ has one simple zero in $\overline{\mathbb{C}^+}$, say $k=ik_1$, 
$k_1>0$, and $a_2(k)$ has one simple zero in $\overline{\mathbb{C}^-}$ at  $k=0$.
Thus we assume that 
 $\dot a_2(0)\ne 0$ and, additionally, we suppose that  $a_{11}:=\lim\limits_{k\to 0}ka_1(k)\not=0$.
\end{description}

\begin{remark}
Case I corresponds to the inequality
$
|v_2(0,0)|^2-|v_1(0,0)|^2 \ne 0
$
whereas in Case II the equality $
|v_2(0,0)|^2-|v_1(0,0)|^2 = 0
$ holds, see (\ref{vv}) and (\ref{sz}).
With this respect, Case I corresponds to ``generic'' initial conditions
whereas Case II corresponds to ``non-generic'' ones.
\end{remark}

\begin{remark}
From the symmetry relations (\ref{a-sym}) it follows that  
$a_{11}$ is purely imaginary. 
Moreover, if $a_1(k)$ has one simple zero, then
 $\Im a_{11}<0$ in  Case II.
\end{remark}

It is interesting that in contrast with the case of the local NLS,
  the  value of $k_1$ can't be prescribed independently of $b(k)$.

\begin{proposition}\label{zero-nongeneric}
Given $b(k)$ for $k\in\mathbb{R}\setminus\{0\}$,
 the zero $k=ik_1$ of $a_1(k)$ is determined as follows:
\begin{enumerate}[(i)]
	\item 
	In Case I, 
\begin{equation}\label{k1-gen}
k_1=\frac{A}{2}\exp\left\{-\frac{1}{2\pi
i}\mathrm{v.p.}\int_{-\infty}^{\infty}
\frac{\ln\frac{\zeta^2}{\zeta^2+1}
	(1-b(\zeta)\bar{b}(-\zeta))}{\zeta}\,d\zeta
\right\},
\end{equation}
\item
In Case II,
\begin{equation}\label{k1-ngen}
k_1=A\frac{\sqrt{(\Re b(0))^2+E_2^2}-\Re b(0)}{2E_1E_2},
\end{equation}
where
\begin{equation}\label{E1-E2}
E_1=\exp\left\{\frac{1}{2\pi i}\mathrm{v.p.}\int_{-\infty}^{\infty}
\frac{\ln(1-b(\zeta)\bar{b}(-\zeta))}{\zeta}\,d\zeta
\right\} \quad \text{and} \quad E_2=\exp\left\{\frac{1}{2}\ln(1-|b(0)|^2)\right\}
\end{equation}
(notice that $1-|b(0)|^2=a_{11}\dot a_2(0)\ne 0$ by assumption).
\end{enumerate}

\end{proposition}
\begin{proof} (i) \textit{Case I}. 
Define functions $\tilde a_1(k)$ and $\tilde a_2(k)$ by 
 \[
\tilde a_1(k) = a_1(k)\frac{k^2}{(k-ik_1)(k+i)}, \qquad 
\tilde a_2(k) = a_2(k)\frac{k-ik_1}{k-i}.
\]
Then the determinant relation (see Item 4 in Proposition \ref{properties})
can be viewed as the  following scalar RH problem w.r.t.  $\tilde a_j(k)$, $j=1,2$:
given $b(k)$, $k\in \mathbb R$, find $\tilde a_1(k)$ and $\tilde a_2(k)$
analytic and having no zeros in $\overline{\mathbb{C}^+}$ and $\overline{\mathbb{C}^-}$ respectively, satisfying the jump condition
\begin{equation}\label{sRH}
\tilde a_1(k) \tilde a_2(k) =
\frac{k^2}{k^2+1}(1-b(k)\bar{b}(-k)),\quad k\in\mathbb{R}
\end{equation}
and  the normalization conditions $\tilde a_j(k)\to 1$ as $k\to\infty$.
The unique solution of this RH problem is given by 
\[
\tilde a_1(k) = e^{\chi(k)},\qquad \tilde a_2(k) = e^{-\chi(k)},
\]
where 
\[
\chi(k) = \frac{1}{2\pi i}
	\int_{-\infty}^{\infty}\frac{\ln
		\frac{\zeta^2}{\zeta^2+1}(1-b(\zeta)\bar{b}(-\zeta))}
	{\zeta-k}\,d\zeta.
\]
Then $a_1(k)$ and $a_2(k)$ can be written as
\begin{subequations}\label{traceg}
	\begin{equation}
	\label{traceg1}
	a_1(k)=\frac{(k-ik_1)(k+i)}{k^2}e^{\chi(k)}
		\end{equation}
	and
	\begin{equation}
	a_2(k)=\frac{k-i}{k-ik_1}e^{-\chi(k)},
		\end{equation}
\end{subequations}
which, being evaluated  at $k=0$, gives
\begin{equation}\label{a-12-1}
a_1(k)=\frac{k_1e^{\chi(+i0)}}{k^2}(1+o(k)) \ \ \text{and}\ \ 
a_2(0)=\frac{e^{-\chi(-i0)}}{k_1}.
\end{equation}
On the other hand (see (\ref{sz})), 
\begin{equation}\label{a-12-2}
a_1(k)=\frac{A^2a_2(0)}{4k^2}(1+o(k)), \qquad k\to 0.
\end{equation}
Comparing (\ref{a-12-1}) and (\ref{a-12-2}) and taking into account that 
(by the Sokhotski-Plemelj formulas)
\[
\chi(+i0) + \chi(-i0) = \frac{1}{\pi i}\mathrm{v.p.}\int_{-\infty}^{\infty}
\frac{\ln\frac{\zeta^2}{\zeta^2+1}
(1-b(\zeta)\bar{b}(-\zeta))}{\zeta}\,d\zeta,
\]
we arrive at (\ref{k1-gen}).

(ii) \textit{Case II.}
Observe that due to the symmetry relation (\ref{7.9}) and Item (vi) in  Proposition \ref{prop1}, the behavior of  $\Psi_j(x,t,k)$, $j=1,2$ as $k\to 0$ can be characterized as follows:
\begin{subequations}
\label{k0s}
\begin{align}
\label{k0s-a}
&\Psi_1^{(1)}(x,t,k)=\frac{1}{k}
\begin{pmatrix}v_1(x,t)\\v_2(x,t)\end{pmatrix}
+\begin{pmatrix} s_1(x,t)\\ s_2(x,t)\end{pmatrix}
+O(k),\\
&\Psi_1^{(2)}(x,t,k)=\frac{2i}{A}\begin{pmatrix}v_1(x,t)\\ v_2(x,t)\end{pmatrix}
+k\begin{pmatrix} h_1(x,t)\\ h_2(x,t)\end{pmatrix}
+O(k^2),\\
\label{k0s-b}
&\Psi_2^{(1)}(x,t,k)=-\frac{2i}{A}
\begin{pmatrix}
\overline{v_2}(-x,t)\\ \overline{v_1}(-x,t)
\end{pmatrix}
-k\begin{pmatrix}\overline{h_2}(-x,t)\\ \overline{h_1}(-x,t) \end{pmatrix}
+O(k^2),\\
&\Psi_2^{(2)}(x,t,k)=
-\frac{1}{k}
\begin{pmatrix}\overline{v_2}(-x,t)\\\overline{v_1}(-x,t)
\end{pmatrix}
+\begin{pmatrix} \overline{s_2}(-x,t)\\ \overline{s_1}(-x,t)\end{pmatrix}
+O(k),
\end{align}
\end{subequations}
with some $v_j$, $s_j$, and $h_j$ ($j=1,2$).
Then, using the definitions (\ref{sd}) of the spectral functions and taking into account that $|v_2(0,0)|^2-|v_1(0,0)|^2=0$ in Case II, we have as $k\to 0$:
\begin{subequations}
\label{zzz}
\begin{align}
\label{zzza}
a_1(k)&=
\left.\frac{1}{k}(v_1\bar{s}_1-\bar{v}_1s_1-v_2\bar{s}_2+\bar{v}_2s_2)
\right|_{x,t=0}+O(1),\\
a_2(k)&=
\left.k\frac{2i}{A}
(v_1\bar{h}_1+\bar{v}_1h_1-v_2\bar{h}_2-\bar{v}_2h_2)\right|_{x,t=0}
+O(k^2),\\
b(k)&=
\left.v_1\bar{h}_1-v_2\bar{h}_2+\frac{2i}{A}(\bar{v}_1s_1-\bar{v}_2s_2)
\right|_{x,t=0}+O(k).
\end{align}
\end{subequations}
Equations (\ref{zzz}) imply that
\begin{equation}\label{aab}
a_{11}=iA\Re b(0)-\frac{A^2}{4}\dot a_2(0),
\end{equation}
where $a_{11}=\lim\limits_{k\to0}(ka_1(k))$.

On the other hand, 
introducing 
\[
\hat a_1(k)=a_1(k) \frac{k}{k-ik_1} \quad \text{and} \ \ 
\hat a_2(k)=a_2(k) \frac{k-ik_1}{k},
\]
 the determinant relation can be viewed as the scalar RH problem with the jump condition
\[
\hat a_1(k)\hat a_2(k) = 1-b(k)\bar b(-k),
\]
whose solution gives
\begin{subequations}\label{traceng}
\begin{equation}
\label{tracenga}
a_1(k)=\frac{k-ik_1}{k}\exp\left\{\frac{1}{2\pi i}
\int_{-\infty}^{\infty}\frac{\ln(1-b(\zeta)\bar{b}(-\zeta))}
{\zeta-k}\,d\zeta\right\},
\end{equation}
and
\begin{equation}
a_2(k)=\frac{k}{k-ik_1}\exp\left\{-\frac{1}{2\pi i}
\int_{-\infty}^{\infty}\frac{\ln(1-b(\zeta)\bar{b}(-\zeta))}
{\zeta-k}\,d\zeta\right\}.
\end{equation}
\end{subequations}
From  (\ref{traceng}), using the Sokhotski-Plemelj formulas, we obtain
\begin{equation}
a_{11}=-ik_1E_1E_2 \quad \text{and}\ \  \dot a_2(0)=\frac{i}{k_1}E_1^{-1}E_2,
\end{equation}
where $E_1$ and $E_2$ are given by (\ref{E1-E2}),
which, being compared with 
 (\ref{aab}), uniquely determines $k_1>0$  as the solution of a quadratic equation.
\end{proof}

Taking into account the singularities of 
 $\Psi_j(x,t,k)$, $j=1,2$ and  $a_1(k)$  at  $k=0$ (see Proposition \ref{prop1}),  the behavior of $M(x,t,k)$ at $k=0$ can be described as follows: in Case I,
\begin{subequations}\label{z}
\begin{align}
\label{+i0}
& M_+(x,t,k)=
\begin{pmatrix}
\frac{4}{A^2a_2(0)}v_1(x,t)& -\overline{v_2}(-x,t)\\
\frac{4}{A^2a_2(0)}v_2(x,t)& -\overline{v_1}(-x,t)
\end{pmatrix}
(I+O(k))
\begin{pmatrix}
k& 0\\
0& \frac{1}{k}
\end{pmatrix}, & k\rightarrow +i0,\\
\label{-i0}
& M_-(x,t,k)=\frac{2i}{A}
\begin{pmatrix}
-\overline{v_2}(-x,t)& \frac{v_1(x,t)}{a_2(0)}\\
-\overline{v_1}(-x,t)& \frac{v_2(x,t)}{a_2(0)}
\end{pmatrix}
+O(k), & k\rightarrow -i0,
\end{align}
\end{subequations}
and Case II,
\begin{subequations}\label{nz}
\begin{align}
\label{n+i0}
& M_+(x,t,k)=
\begin{pmatrix}
\frac{v_1(x,t)}{a_{11}}& -\overline{v_2}(-x,t)\\
\frac{v_2(x,t)}{a_{11}}& -\overline{v_1}(-x,t)
\end{pmatrix}
(I+O(k))
\begin{pmatrix}
1& 0\\
0& \frac{1}{k}
\end{pmatrix}, & k\rightarrow +i0,\\
\label{n-i0}
& M_-(x,t,k)=\frac{2i}{A}
\begin{pmatrix}
-\overline{v_2}(-x,t)& \frac{v_1(x,t)}{\dot{a}_2(0)}\\
-\overline{v_1}(-x,t)& \frac{v_2(x,t)}{\dot{a}_2(0)}
\end{pmatrix}
(I+O(k))
\begin{pmatrix}
1& 0\\
0& \frac{1}{k}
\end{pmatrix}, & k\rightarrow -i0
\end{align}
\end{subequations}
(recall that $a_{11}$ is determined by  $a_1(k)=\frac{a_{11}}{k}+O(1)$ as $k\to 0$).

Additionally, if $a_1(ik_1)=0$  with $k_1>0$
(recall that in this case we assume that this zero is simple), then   $M(x,t,k)$ satisfies the  residue condition 
\begin{equation}
\label{res}
\underset{k=ik_1}{\operatorname{Res}} M^{(1)}(x,t,k)=
\frac{\gamma_1}{\dot{a}_1(ik_1)}e^{-2k_1x-4ik_1^2t}M^{(2)}(x,t,ik_1),\quad
|\gamma_1|=1,
\end{equation}
where $\Psi_1^{(1)}(0,0,ik_1)=\gamma_1\Psi_2^{(2)}(0,0,ik_1)$. Notice that the symmetry relation (\ref{7.9}) implies that 
$\overline{\Psi}_1^{(1)}(0,0,ik_1)=
\gamma_1^{-1}\overline{\Psi}_2^{(2)}(0,0,ik_1)$ and thus $|\gamma_1|=1$
(cf. \cite{AMN}).

Notice that if $a_1(k)$ has a  zero $k=\zeta_1$ that is not pure imaginary, then,
due to the symmetry conditions, it also has a zero at $k=\zeta_2=-\bar \zeta_1$,
and the associated residue conditions have the form:
\begin{subequations}\label{2res}
\begin{equation}
\label{2res1}
\underset{k=\zeta_1}{\operatorname{Res}} M^{(1)}(x,t,k)=
\frac{\eta_1}{\dot a_1(\zeta_1)}
e^{2i\zeta_1x+4i\zeta_1^2t}M^{(2)}(x,t,\zeta_1)
\end{equation}
and
\begin{equation}
\label{2res2}
\underset{k=\zeta_2}{\operatorname{Res}} M^{(1)}(x,t,k)=
\frac{1}{\bar\eta_1\dot a_1(\zeta_2)}
e^{2i\zeta_2x+4i\zeta_2^2t}M^{(2)}(x,t,\zeta_2),
\end{equation}
\end{subequations}
where $\eta_1$ is determined by $\Psi_1^{(1)}(0,0,\zeta_1)=\eta_1\Psi_2^{(2)}(0,0,\zeta_1)$.

Now we are at a position to formulate the Riemann-Hilbert problem, whose solution gives 
the solution of the initial value  problem (\ref{1}), (\ref{2.0}).
Let $b(k)$, $k\in{\mathbb R}$
and $\gamma_1$ with $|\gamma_1|=1$ be the spectral data associated with the initial data $q_0(x)$
in (\ref{1}). Then the Riemann-Hilbert problem is as follows:
\begin{description}
\item [Basic Riemann--Hilbert Problem:] 
Given $b(k)$ and $\gamma_1$,
find the $2\times 2$-valued function $M(x,t,k)$, piece-wise meromorphic in $k$ relative to  
$\mathbb{R}$ and satisfying the following conditions:

\begin{description}
\item [(i)] Jump conditions. The  non-tangential limits $M_{\pm}(x,t,k)=M(x,t,k\pm i0)$
exist a.e. for $k\in\mathbb{R}$ such that $M(x,t,\cdot)-I \in L^2(\mathbb{R}\setminus[-\varepsilon, \varepsilon])$ for any $\varepsilon>0$ and 
$M_{\pm}(x,t,k)$ satisfy the condition
\begin{equation}\label{jRH}
M_+(x,t,k)=M_-(x,t,k)J(x,t,k) \qquad \text{for a.e.}\  k\in\mathbb{R}\setminus\{0\},
\end{equation}
where the jump matrix $J(x,t,k)$ is given by (\ref{jump}), with $r_1$ and $r_2$ given
in terms of $b$ 
by (\ref{r12})
with (\ref{traceg}) (Case I) or (\ref{traceng}) (Case II).
\item[(ii)] Normalization at $k=\infty$:
$$
M(x,t,k)=I+O(k^{-1})\qquad \text{uniformly as}\ k\to\infty.
$$
\item [(iii)] Residue condition (\ref{res}) with $k_1$ given in terms of $b$ using (\ref{k1-gen}) (Case I) or (\ref{k1-ngen}) (Case II).
\item [(iv)] Singularity conditions at $k=0$: $M(x,t,k)$ satisfies 
(\ref{z}) (Case I) or (\ref{nz}) (Case II), 
where $v_j(x,t)$, $j=1,2$ are some (not prescribed) functions.

\end{description}

Assume that the RH problem (i)--(iv) has a solution $M(x,t,k)$. 
Then the solution of the initial value problem (\ref{1}), (\ref{2.0}) 
is given in terms of the (12) and (21) entries of $M(x,t,k)$ as follows:
\begin{equation}\label{sol}
q(x,t)=2i\lim_{k\to\infty}kM_{12}(x,t,k),
\end{equation}
and
\begin{equation}\label{sol1}
q(-x,t)=-2i\lim_{k\to\infty}k\overline{M_{21}(x,t,k)}.
\end{equation}

\end{description}

The solution of the RH problem is unique, if exists. Indeed, if $M$ and $\tilde M$ are 
two solutions, then
conditions (\ref{z})  or (\ref{nz}) provide the boundedness of $M \tilde M^{-1}$ at $k=0$.
Then 
the standard arguments based of the Liouville theorem leads to $M \tilde M^{-1}\equiv I$.

\begin{remark} 
From  (\ref{sol}) and (\ref{sol1}) it follows  that in order to present the 
solution of (\ref{1}), (\ref{2.0}) for all $x\in \mathbb R$, it is sufficient to have the solution of
 the RH problem for, say, $x\ge 0 $ only.
\end{remark}

\begin{remark}
In the general case with more zeros of $a_1(k)$ in ${\mathbb C}^+$ and/or zeros of  $a_2(k)$
in ${\mathbb C}^-$, relevant residue conditions, of type (\ref{res}) and/or (\ref{2res}),
have to be specified, in terms of a prescribed set of zeros and corresponding norming constants.
\end{remark}

\begin{proposition}\label{lemma-sym}
The solution $M$ of the Riemann--Hilbert problem \textbf{(i)-(iv)}
satisfies the following symmetry condition (cf. (\ref{phi-sym})):
\begin{equation}
M(x,t,k)=\begin{cases}
\Lambda \overline{M(-x,t,-\bar k)} \Lambda^{-1} 
	\begin{pmatrix}
		\frac{1}{a_1(k)} & 0 \\ 0 & a_1(k)
	\end{pmatrix}, & k\in{\mathbb C}^+\setminus\{0\}, \\
	\Lambda \overline{M(-x,t,-\bar k)} \Lambda^{-1}
	\begin{pmatrix}
		a_2(k) & 0 \\ 0 & \frac{1}{a_2(k)}
	\end{pmatrix}, & k\in{\mathbb C}^-\setminus\{0\}.
\end{cases}
\label{M-sym}
\end{equation}
\end{proposition}
\begin{proof} 
Follows from   the symmetry of  the jump matrix (\ref{jump})
in (\ref{jRH}) 
\[
\Lambda \overline{J(-x,t,-k)} \Lambda^{-1} = \begin{pmatrix}
		a_2(k) & 0 \\ 0 & \frac{1}{a_2(k)}\end{pmatrix} J(x,t,k)\begin{pmatrix}
		a_1(k) & 0 \\ 0 & \frac{1}{a_1(k)}
	\end{pmatrix},\qquad k\in\mathbb{R}\setminus\{0\}
\]
(which, in turns, follows from (\ref{r-sym}) and (\ref{r-a})), and the fact that the structural conditions (\ref{z}) and (\ref{nz}) and the residue condition (\ref{res}) are
 consistent with (\ref{M-sym}).
\end{proof}

\subsection{One-soliton solution}

\begin{proposition} 
Let $a_1(k)$, $a_2(k)$, and $b(k)$ be the spectral functions (i) associated with some $q_0(x)$ and (ii) satisfying the following conditions:
\begin{itemize}
	\item 
	$b(k)=0$ for all $k\in\mathbb{R}$;
	\item
	$a_1(k)$ has a single, simple zero $k=ik_1$ with some $k_1>0$ in $\overline{\mathbb C^+}$;
	\item
	$a_2(k)$ has a single, simple zero $k=0$ in $\overline{\mathbb C^-}$.
\end{itemize} 
Also, let $\gamma_1$ be given such that 
	$\gamma_1=e^{i\phi_1}$ with $\phi_1\in \mathbb R$.
	Then: 
	\begin{enumerate}[(1)]
		\item 
		$k_1$ is uniquely determined as $k_1=\frac{A}{2}$;
		\item
		The Riemann--Hilbert problem (i)--(iv) has a unique solution
		for all $(x,t)$ with $x\in \mathbb R$ and $t\ge 0$ except the set 
		$\cup_{n\in \mathbb Z}\{(0,t_n)\}$ with $t_n=\frac{\phi_1}{A^2} + \frac{2\pi}{A^2}n$;
		\item
		The associated exact solution $q(x,t)$ of problem (\ref{1}), (\ref{2.0}) is given by 
\begin{equation}
		\label{pure-sol}
q(x,t)=\frac{A}{1-e^{-Ax-iA^2t+i\phi_1}}.
\end{equation}
	\end{enumerate}
	
	\end{proposition}
	
\begin{proof}
Since $b(0)=0$, we are in Case II, and thus Item 1 follows 
from Proposition \ref{zero-nongeneric}, (ii). 
Moreover, (\ref{traceng}) gives 

\begin{equation}
\label{sp}
a_1(k)=\frac{k-i\frac{A}{2}}{k},\qquad a_2(k)=\frac{k}{k-i\frac{A}{2}}
\end{equation}
and thus the constants involved in (\ref{nz}) are as follows:
\[
a_{11}=\frac{A}{2i}, \qquad \dot a_2(0)=\frac{2i}{A}.
\]

Now notice that since $b(k)\equiv 0$, it follows that $M(\cdot, \cdot, k)$ is a meromorphic 
(in $\mathbb C$) function with the only pole at $k=ik_1$. Then, 
comparing (\ref{n+i0}) and (\ref{n-i0}), it follows that 
 $v_1(x,t)=-\bar v_2(-x,t)$ and thus
the singularity  conditions (\ref{nz}) reduce to a conventional residue condition:
\begin{equation}
\label{res0}
\underset{k=0}{\operatorname{Res}}\  M^{(2)}(x,t,k)=
\frac{A}{2i} M^{(1)}(x,t,0).
\end{equation}
Further, taking into account the original residue condition (\ref{res})
and the normalization condition (ii), we arrive at the 
following representation for $M$:
 \begin{equation}
\label{M-sol}
M(x,t,k) = \begin{pmatrix}
\frac{k+v_1(x,t)}{k-\frac{iA}{2}}  &  \frac{v_1(x,t)}{k} \\
\frac{-\bar v_1(-x,t)}{k-\frac{iA}{2}} & \frac{k-\bar v_1(-x,t)}{k}
\end{pmatrix},
\end{equation}
where $v_1(x,t)$ is determined using (\ref{res}):
\begin{equation}
\label{v1-sol}
v_1(x,t)=\frac{A}{2i} \frac{1}{1-e^{-Ax-iA^2t+i\phi_1}}.
\end{equation}
Particularly, this determines the singularity set as the set of zeros of the denominator in 
(\ref{v1-sol}). 
Finally, using (\ref{sol}) or (\ref{sol1}), the soliton formula (\ref{pure-sol}) follows.
\end{proof}

\section{The long-time asymptotics}

The shock-type long-time asymptotics for the local NLS equation with the step-like
boundary conditions (\ref{2}), (\ref{3}) was presented in \cite{BKS}, 
where it was shown that there were always three sectors in the $(x,t)$ half-plane ($t>0$)
characterized by qualitatively different asymptotic behavior: the decaying sector
(where the order of decay of $q$ is $O(t^{-1/2})$), the sector of modulated elliptic wave,
and the sector of modulated plane wave. Particularly, if $B=0$, then the modulated elliptic wave
occupies the sector $0<\frac{x}{t} < 8\sqrt{2} A$. 

It is natural to compare this behavior with the asymptotics for the nonlocal NLS equation
with the same type of the initial data (\ref{1-b}), (\ref{1-c}).
This motivate us to study, in  this Section, the long-time asymptotics of the solution of the initial value
 problem (\ref{1}),
(\ref{2.0}). Our analysis is based on the adaptation of 
 the nonlinear steepest-decent method \cite{DZ} to the (oscillatory) RH problem (i)--(iv). 
The implementation of the method in our case has some specific features: particularly, 
we have to deal with a singularity on the contour,  and the jump $1+r_1(k)r_2(k)$ in the scalar RH problem for $\delta(\xi,k)$ (see (\ref{13}) below) is not, in general, real-valued.

We will show that a basic difference of the asymptotics for the nonlocal NLS equation
being compared with that for the local NLS is that, while there are still the sector of decay
and the sector of ``modulated constant'', there is 
no an intermediate sector  between these two (although a transition zone
between these sectors may exist, being characterized by a specific asymptotics along curves converging to the ray $x=0$, $t>0$).

\subsection{Jump factorizations}

First, notice that in view of  (\ref{sol}) and (\ref{sol1}), 
studying the RH problem for $x>0$
is sufficient for studying $q(x,t)$ for all $(x,t)$ outside the sector $|x/t|<\varepsilon$
for any $\varepsilon>0$.

Introduce the  variable $\xi:=\frac{x}{4t}$ and the phase function
\begin{equation}\label{theta}
\theta(k,\xi)=4k\xi+2k^2.
\end{equation}
The jump matrix (\ref{jump}) allows, similarly to  \cite{RS}, two triangular factorizations:
\begin{subequations}\label{tr}
\begin{equation}
\label{tr1}
J(x,t,k)=
\begin{pmatrix}
1& 0\\
\frac{r_1(k)}{1+ r_1(k)r_2(k)}e^{2it\theta}& 1\\
\end{pmatrix}
\begin{pmatrix}
1+ r_1(k)r_2(k)& 0\\
0& \frac{1}{1+ r_1(k)r_2(k)}\\
\end{pmatrix}
\begin{pmatrix}
1& \frac{ r_2(k)}{1+ r_1(k)r_2(k)}e^{-2it\theta}\\
0& 1\\
\end{pmatrix}
\end{equation}
and
\begin{equation}
\label{tr2}
J(x,t,k)=
\begin{pmatrix}
1& r_2(k)e^{-2it\theta}\\
0& 1\\
\end{pmatrix}
\begin{pmatrix}
1& 0\\
r_1(k)e^{2it\theta}& 1\\
\end{pmatrix}.
\end{equation}
\end{subequations}
Since the  phase function $\theta(k,\xi)$ is the same as in the case of the local NLS,
its  signature table (see Figure \ref{signtable}) suggests us to follow the conventional
 steps \cite{DZ, DIZ}
involving (i) getting rid of the diagonal factor in (\ref{tr1}) and (ii) 
the deformation of the 
original RH problem (relative to the real axis) to a new one, relative to a cross, where the jump matrix converges,
as $t\to\infty$,
to the identity matrix uniformly away from any vicinity of the stationary phase point $k=-\xi$.
But when following this scheme, 
we have to pay a special attention to the singularity point  $k=0$.

\begin{figure}
\centering{\includegraphics[scale=0.3]{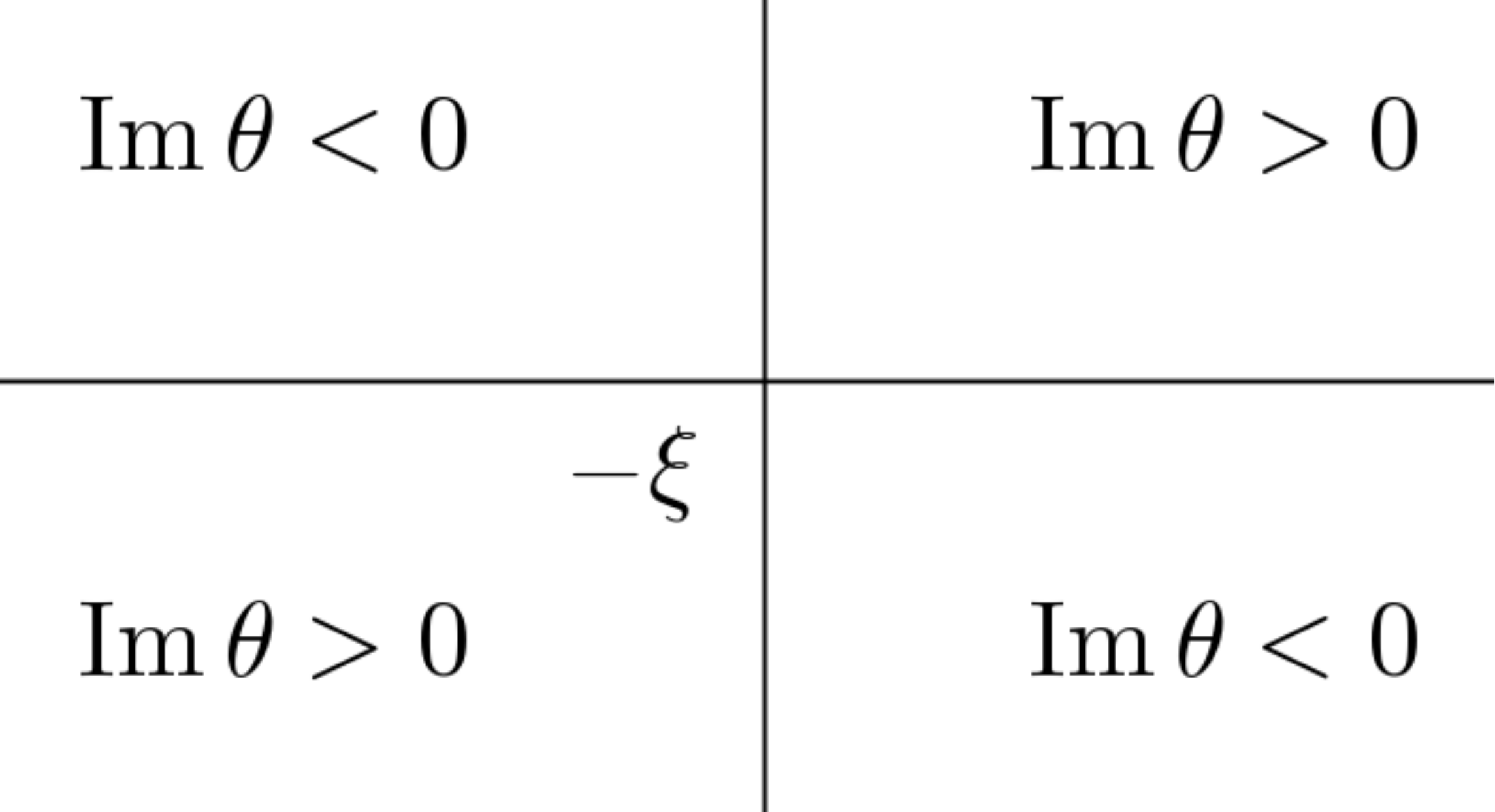}}
\caption{Signature table}
\label{signtable}
\end{figure}

First, introduce $\delta(\xi,k)$ as the solution of the scalar RH problem: find
$\delta(\xi,k)$ analytic in ${\mathbb C}\setminus (-\infty,-\xi]$ and satisfying 
the conditions
\begin{equation}
\label{13}
\begin{cases}
\delta_+(\xi,k)=\delta_-(\xi,k)(1+r_1(k)r_2(k)),\,& k\in(-\infty,-\xi),\\
\delta(\xi,k)\rightarrow 1, \,& k\rightarrow\infty.
\end{cases}
\end{equation}
Its  solution  is given by the Cauchy-type integral:
\begin{equation}
\label{14}
\delta(\xi,k)=\exp\left\{\frac{1}{2\pi i}\int_{-\infty}^{-\xi}\frac{\ln(1+r_1(\zeta)r_2(\zeta))}{\zeta-k}\,d\zeta\right\}
\end{equation}
(notice that since we deal with $\xi>0$, the behavior of $r_j(k)$ at $k=0$ does not affect 
$\delta(\xi,k)$). 
Then define $\tilde M$ with the help of $\delta$:
\begin{equation}\label{ft}
\tilde{M}(x,t,k)=M(x,t,k)\delta^{-\sigma_3}(\xi,k).
\end{equation}

Notice that in the case of the pure-step initial data (\ref{step-ini}), $1+r_1(k)r_2(k)=\frac{4k^2}{4k^2+A^2}$ (see Remark \ref{r1}), and thus $1+r_1(k)r_2(k)$  is real-valued. However, in the general case, 
$1+r_1(k)r_2(k)$ can take complex values, which may cause $\delta(\xi,k)$ to be singular at 
$k=-\xi$ (cf. \cite{RS}).

Indeed, $\delta(\xi,k)$ can be written as 
\begin{equation}\label{delta-singular}
\delta(\xi,k)=
(\xi+k)^{i\nu(-\xi)}e^{\chi(\xi,k)},
\end{equation}
where
\begin{equation}\label{as5.5}
\chi(\xi,k):=-\frac{1}{2\pi i}\int_{-\infty}^{-\xi}\ln(k-\zeta)d_{\zeta}\ln(1+ r_1(\zeta)r_2(\zeta))
\end{equation}
and 
\begin{equation}\label{nu}
\nu(-\xi):=-\frac{1}{2\pi}\ln(1+r_1(-\xi)r_2(-\xi))
= -\frac{1}{2\pi}\ln|1+r_1(-\xi)r_2(-\xi)|-
\frac{i}{2\pi}\Delta(-\xi),
\end{equation}
with 
\[
\Delta(-\xi):=\int_{-\infty}^{-\xi}d\arg(1+ r_1(\zeta)r_2(\zeta)).
\]

In what follows we will assume 
 that 
\begin{equation}\label{arg-ass}
\Delta(k)\in(-\pi,\pi)\qquad \text{for all}\ \ k\in(-\infty,0)
\end{equation}
and thus $|\Im \nu(k)|<\frac{1}{2}$.
In this case, $\ln(1+r_1(k)r_2(k))$ is single-valued, and  the singularity
of $\delta(\xi,k)$ (as well as of $\tilde M(x,t,k)$) at $k=-\xi$ is square integrable. 
More importantly, assumption (\ref{arg-ass}) will allow us to establish correct 
estimates, see (\ref{as-sol}) in Theorem \ref{th1}, i.e. the estimates 
with main terms dominating  the 
error ones.

Assumption (\ref{arg-ass}) obviously holds 
in the case of the pure-step initial data (\ref{step-ini}):
 in this case, 
$\Delta(k)\equiv0$ for $k\in(-\infty,0)$.
With this respect, this assumption holds, particularly, 
if the initial data are small $L^1$-perturbations of $q_{0A}(x)$;
we have already remarked  on this aspect when formulating the 
conditions for Case I and Case II above.

Function $\tilde{M}(x,t,k)$ defined by (\ref{ft}) satisfies the  RH problem specified by  
the jump, normalization, and residue conditions:  
\begin{subequations}\label{tildeM}
\begin{equation}
\label{14.1}
\tilde{M}_+(x,t,k)=\tilde{M}_-(x,t,k)\tilde{J}(x,t,k), \qquad k\in\mathbb{R}\setminus\{0\},
\end{equation}
\begin{equation}
\label{14.1-norm}
\tilde{M}(x,t,k)\rightarrow I, \qquad k\rightarrow\infty,
\end{equation}
\begin{equation}
\label{14.4}
\underset{k=ik_1}{\operatorname{Res}} \tilde{M}^{(1)}(x,t,k)=
\frac{\gamma_1}{\dot{a}_1(ik_1)\delta^{2}(\xi,ik_1)}
e^{-2k_1x-4ik_1^2t}\tilde{M}^{(2)}(x,t,ik_1),\quad |\gamma_1|=1,
\end{equation}
where 
\begin{equation}
\label{as4}
\tilde{J}(x,t,k)=
\begin{cases}
\begin{pmatrix}
1& 0\\
\frac{r_1(k)\delta_-^{-2}(\xi,k)}{1+r_1(k)r_2(k)}e^{2it\theta}& 1\\
\end{pmatrix}
\begin{pmatrix}
1& \frac{r_2(k)\delta_+^{2}(\xi,k)}{1+r_1(k)r_2(k)}e^{-2it\theta}\\
0& 1\\
\end{pmatrix},\,& k\in(-\infty,-\xi),
\\
\begin{pmatrix}
1& r_2(k)\delta^2(\xi,k)e^{-2it\theta}\\
0& 1\\
\end{pmatrix}
\begin{pmatrix}
1& 0\\
r_1(k)\delta^{-2}(\xi,k)e^{2it\theta}& 1\\
\end{pmatrix},\,& k\in(-\xi,\infty)\setminus\{0\},
\end{cases}
\end{equation}
supplemented by the singularity conditions at $k=0$:
\begin{align}
\label{14.2}
&\tilde{M}_+(x,t,k)=
\begin{pmatrix}
\frac{4v_1(x,t)}{A^2a_2(0)\delta(\xi,0)}& -\delta(\xi,0)\overline{v_2}(-x,t)\\
\frac{4v_2(x,t)}{A^2a_2(0)\delta(\xi,0)}& -\delta(\xi,0)\overline{v_1}(-x,t)
\end{pmatrix}
(I+O(k))
\begin{pmatrix}
k& 0\\
0& \frac{1}{k}
\end{pmatrix}, &
k\rightarrow +i0,\\
\label{14.3}
&\tilde{M}_-(x,t,k)=\frac{2i}{A}
\begin{pmatrix}
\frac{-\overline{v_2}(-x,t)}{\delta(\xi,0)}& \delta(\xi,0)\frac{v_1(x,t)}{a_2(0)}\\
\frac{-\overline{v_1}(-x,t)}{\delta(\xi,0)}& \delta(\xi,0)\frac{v_2(x,t)}{a_2(0)}
\end{pmatrix}
+O(k), &
k\rightarrow -i0,
\end{align}
in Case I, and 
\begin{align}
& \tilde M_+(x,t,k)=
\begin{pmatrix}
\frac{v_1(x,t)}{a_{11}\delta(\xi,0)}& 
-\delta(\xi,0)\overline{v_2}(-x,t)\\
\frac{v_2(x,t)}{a_{11}\delta(\xi,0)}& 
-\delta(\xi,0)\overline{v_1}(-x,t)
\end{pmatrix}
(I+O(k))
\begin{pmatrix}
1& 0\\
0& \frac{1}{k}
\end{pmatrix}, & k\rightarrow +i0,\\
& \tilde M_-(x,t,k)=\frac{2i}{A}
\begin{pmatrix}
-\frac{\overline{v_2}(-x,t)}{\delta(\xi,0)}& \delta(\xi,0)\frac{v_1(x,t)}{\dot{a}_2(0)}\\
-\frac{\overline{v_1}(-x,t)}{\delta(\xi,0)}& \delta(\xi,0)\frac{v_2(x,t)}{\dot{a}_2(0)}
\end{pmatrix}
(I+O(k))
\begin{pmatrix}
1& 0\\
0& \frac{1}{k}
\end{pmatrix}, & k\rightarrow -i0,
\end{align}
in Case II.
\end{subequations}

\subsection{RH problem deformations}\label{ss:deform}

Notice that similarly to the case of the NLS equation, 
assuming that $\int_{-\infty}^0 |q_0(x)|dx <\infty$ and 
$\int^\infty_0 |q_0(x)-A|dx <\infty$,
the reflection coefficients
 $r_j(k)$, $j=1,2$, are defined, in general, for $k\in\mathbb{R}$ only
 (see Propositions \ref{prop1} and \ref{properties}).
On the other hand, in the large-$t$ analysis of $\tilde M(x,t,k)$, 
 it is advantageous to 
 have $r_j(k)$ continued, as meromorphic functions, into $\mathbb{C}$;
 then this will allow us to proceed with
 appropriate RH problem deformations.
 Otherwise $r_j(k)$ and  $\frac{r_j(k)}{1+ r_1(k)r_2(k)}$ have to be approximated 
  by some rational functions with well-controlled errors (see, e.g., \cite{DIZ,Len15}).

 For clarity's sake, in what follows we will assume that
 the initial data $q_0(x)$ are a compact
perturbation of the pure step initial data  $q_{0A}(x)$ (\ref{step-ini}), which guarantees 
  that all $\Psi_l^m(x,0,k)$, $l,m=1,2$ (see Proposition \ref{prop1}) and thus $r_j(k)$ are meromorphic 
 in $\mathbb{C}$. 
 Then we 
  define $\hat{M}(x,t,k)$ as follows 
(see Figure \ref{F1}):
\begin{figure}[ht]
\centering{
\includegraphics[scale=0.3]{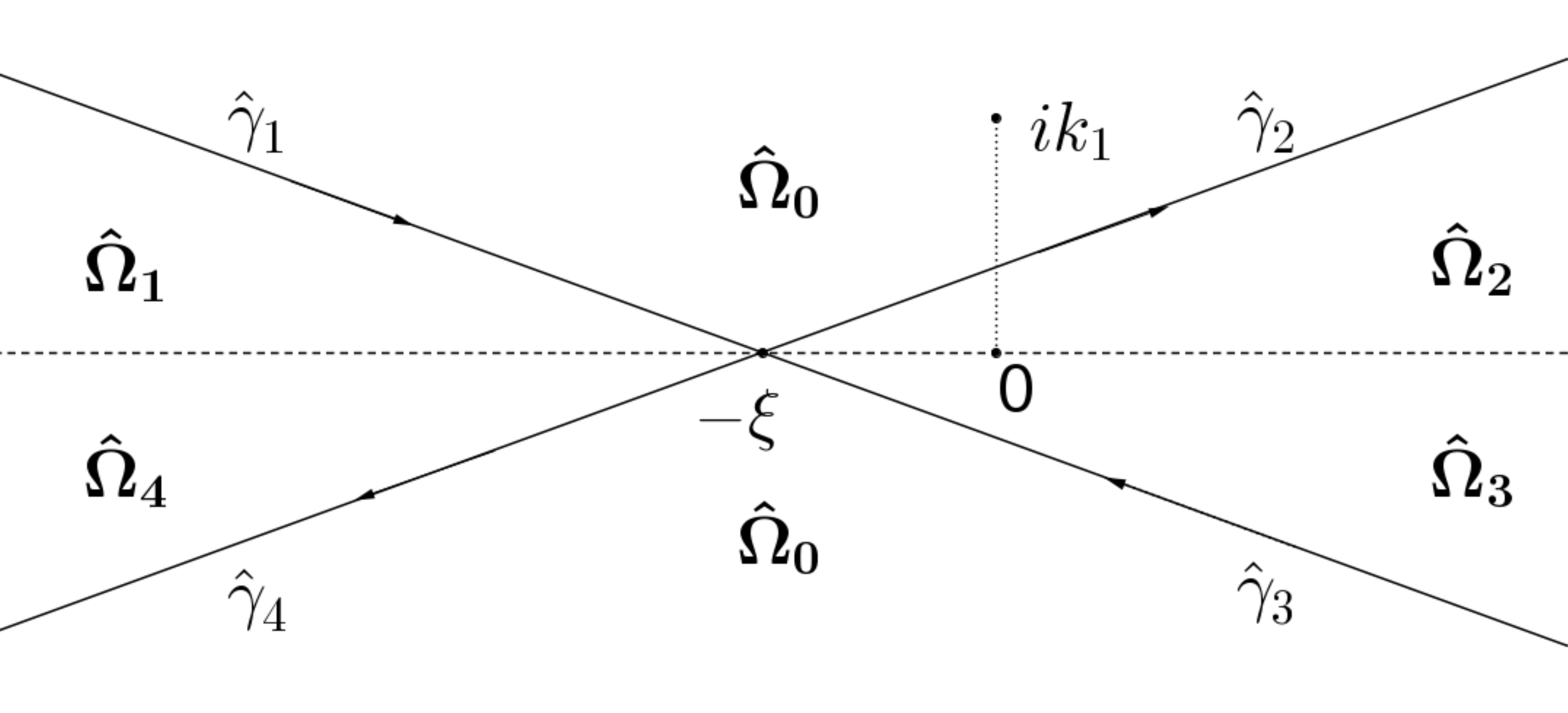}}
\caption{Contour $\hat\Gamma=\hat\gamma_1\cup...\cup\hat\gamma_4$}
\label{F1}
\end{figure}

\begin{equation}
\hat{M}(x,t,k)=
\begin{cases}
\tilde{M}(x,t,k),& k\in\hat\Omega_0,\\
\tilde{M}(x,t,k)
\begin{pmatrix}
1& \frac{-r_2(k)\delta^{2}(\xi,k)}{1+r_1(k)r_2(k)}e^{-2it\theta}\\
0& 1\\
\end{pmatrix}
,& k\in\hat\Omega_1,
\\
\tilde{M}(x,t,k)
\begin{pmatrix}
1& 0\\
-r_1(k)\delta^{-2}(\xi,k)e^{2it\theta}& 1\\
\end{pmatrix}
,& k\in\hat\Omega_2,
\\
\tilde{M}(x,t,k)
\begin{pmatrix}
1& r_2(k)\delta^2(\xi,k)e^{-2it\theta}\\
0& 1\\
\end{pmatrix}
,& k\in\hat\Omega_3,
\\
\tilde{M}(x,t,k)
\begin{pmatrix}
1& 0\\
\frac{r_1(k)\delta^{-2}(\xi,k)}{1+r_1(k)r_2(k)}e^{2it\theta}& 1\\
\end{pmatrix}
,& k\in\hat\Omega_4.
\end{cases}
\end{equation}
Here the angles between the rays $\hat\gamma_j=\hat\gamma_j(\xi)$ and the real axis 
are such that the point $ik_1$ is located in the sector $\hat\Omega_0$.
Then $\hat M(x,t,k)$ satisfies the  RH problem with the jump across 
$\hat\Gamma=\hat\gamma_1\cup...\cup\hat\gamma_4$:
\begin{subequations}\label{RHhat}
\begin{equation}
\label{15.1}
\hat{M}_+(x,t,k)=\hat{M}_-(x,t,k)\hat{J}(x,t,k),\qquad k\in\hat\Gamma
\end{equation}
with 
\begin{equation}
\label{J-hat}
\hat{J}(x,t,k)=
\begin{cases}
\begin{pmatrix}
1& \frac{r_2(k)\delta^{2}(\xi,k)}{1+r_1(k)r_2(k)}e^{-2it\theta}\\
0& 1\\
\end{pmatrix}
,& k\in\hat\gamma_1,
\\
\begin{pmatrix}
1& 0\\
r_1(k)\delta^{-2}(\xi,k)e^{2it\theta}& 1\\
\end{pmatrix}
,& k\in\hat\gamma_2,
\\
\begin{pmatrix}
1& -r_2(k)\delta^2(\xi,k)e^{-2it\theta}\\
0& 1\\
\end{pmatrix}
,& k\in\hat\gamma_3,
\\
\begin{pmatrix}
1& 0\\
\frac{-r_1(k)\delta^{-2}(\xi,k)}{1+r_1(k)r_2(k)}e^{2it\theta}& 1\\
\end{pmatrix}
,& k\in\hat\gamma_4,
\end{cases}
\end{equation}
the normalization 
\begin{equation}
\label{15.2}
\hat{M}(x,t,k)\rightarrow I, \qquad k\rightarrow\infty,
\end{equation}
and the residue condition
\begin{equation}
\label{15.3}
\underset{k=ik_1}{\operatorname{Res}} \hat{M}^{(1)}(x,t,k)=
c_1(x,t)\hat{M}^{(2)}(x,t,ik_1),
\end{equation}
where 
\begin{equation}
\label{c1}
c_1(x,t)=\frac{\gamma_1}{\dot{a}_1(ik_1)\delta^{2}(\xi,ik_1)}
e^{-2k_1x-4ik_1^2t}\qquad \text{ with}\  |\gamma_1|=1.
\end{equation}
As for the singularity conditions at $k=0$, it is remarkable that 
 they reduce, in the both cases, to the same
residue condition having a conventional form
\begin{equation}
\label{res-convent}
\underset{k=0}{\operatorname{Res}}\  \hat{M}^{(2)}(x,t,k)=
c_0(\xi)\hat{M}^{(1)}(x,t,0)
\end{equation}
with (cf. (\ref{res0}))
\begin{equation}
\label{c0}
c_0(\xi)=\frac{A\delta^2(\xi,0)}{2i}.
\end{equation}
\end{subequations}

The  RH problem (\ref{RHhat}) involving two residue conditions (\ref{15.3}) and 
(\ref{res-convent})
can be reduced to a regular RH problem (without residue conditions)
 by using the  Blaschke-Potapov factors (see, e.g., \cite{FT}):

\begin{proposition}
\label{propsol}
The solution $q(x,t)$ of the IV problem (\ref{1}), (\ref{2.0}) 
can be represented as follows:
\begin{subequations}
\label{sol+0}
\begin{align}
\label{sol+0a}
q(x,t)&=-2k_1P_{12}(x,t)+
2i\lim\limits_{k\to\infty}k\hat{M}^{R}_{12}(x,t,k),
\quad x>0,\\
q(x,t)&=-2k_1\overline{P_{21}(-x,t)}
-2i\lim\limits_{k\to\infty}
k\overline{\hat{M}^{R}_{21}(-x,t,k)},\quad x<0.
\end{align}
\end{subequations}
Here (i) $\hat{M}^{R}(x,t,k)$ solves the regular Riemann-Hilbert problem:
\begin{subequations}\label{RHR}
\begin{equation}
\begin{cases}
\hat{M}^R_+(x,t,k)=\hat{M}^R_-(x,t,k)\hat{J}^R(x,t,k),& k\in\hat\Gamma,\\
\hat{M}^R(x,t,k)\rightarrow I, & k\rightarrow\infty,
\end{cases}
\end{equation}
with 
\begin{equation}
\label{J^R}
\hat{J}^R(x,t,k)=
\begin{pmatrix}
1& 0\\
0& \frac{k-ik_1}{k}
\end{pmatrix}
\hat{J}(x,t,k) 
\begin{pmatrix}
1& 0\\
0& \frac{k}{k-ik_1}
\end{pmatrix},\quad k\in\hat{\Gamma}
\end{equation}
\end{subequations}
and (ii) $P_{12}$ and $P_{21}$ are determined in terms of $\hat M^R$:
\begin{equation}
\label{P}
P_{12}(x,t)=\frac{g_1(x,t)h_1(x,t)}
{g_1(x,t)h_2(x,t)-g_2(x,t)h_1(x,t)},\,
P_{21}(x,t)=-\frac{g_2(x,t)h_2(x,t)}
{g_1(x,t)h_2(x,t)-g_2(x,t)h_1(x,t)},
\end{equation}
where $g(x,t)=\left(
\begin{smallmatrix}g_1(x,t)\\g_2(x,t) \end{smallmatrix}\right)$
and  $h(x,t)=\left(
\begin{smallmatrix}h_1(x,t)\\h_2(x,t) \end{smallmatrix}\right)$
are given by
\begin{subequations}\label{gh}
\begin{align}
g(x,t)&=ik_1\hat{M}^{R(1)}(x,t,ik_1)-
c_1(x,t)\hat{M}^{R(2)}(x,t,ik_1),\\
h(x,t)&=ik_1\hat{M}^{R(2)}(x,t,0)+
c_0(\xi)\hat{M}^{R(1)}(x,t,0).
\end{align}
\end{subequations}
\end{proposition}

\begin{proof}
The solution $\hat{M}(x,t,k)$ of the Riemann-Hilbert problem (\ref{RHhat}) 
can be represented in terms of the solution $\hat{M}^R(x,t,k)$ of the regular RH problem
(\ref{RHR}) as follows \cite{FT}:
\begin{equation}
\hat{M}(x,t,k)=
B(x,t,k)\hat{M}^{R}(x,t,k)
\begin{pmatrix}
1& 0\\
0& \frac{k-ik_1}{k}\\
\end{pmatrix},\,k\in\mathbb{C},
\end{equation}
where  the Blaschke-Potapov factor $B$ has the form $B(x,t,k)=I+\frac{ik_1}{k-ik_1}P(x,t)$.
Here  $P(x,t)$ is a projection uniquely determined by the conditions
\begin{equation}
\ker P(x,t)=\mathrm{lin}_{\mathbb{C}}\left\{
g(x,t)\right\}\quad \text{and}\quad
\mathrm{Im}\,P(x,t)=\mathrm{lin}_{\mathbb{C}}\left\{
h(x,t)\right\},
\end{equation}
where $g(x,t)$ and $h(x,t)$ are given by (\ref{gh}): this  implies that 
 the $(12)$ and $(21)$ elements of $P$ are given by (\ref{P}) whereas  
\begin{equation}
P_{11}(x,t)=-\frac{P_{12}(x,t)g_2(x,t)}{g_1(x,t)}\quad\text{and }\ 
P_{22}(x,t)=-\frac{P_{21}(x,t)g_1(x,t)}{g_2(x,t)}.
\end{equation}
Finally, taking into account that 
\begin{equation}
\hat{M}(x,t,k)=
\begin{pmatrix}
1& 0\\
0& 1-\frac{ik_1}{k}
\end{pmatrix}
+\frac{ik_1}{k-ik_1}P(x,t)+\frac{\hat{M}^{R}_1(x,t)}{k}+
O\left(\frac{1}{k^2}\right),\quad k\to\infty
\end{equation}
where $\hat{M}^R(x,t,k)=I+\frac{\hat{M}^R_1(x,t)}{k}+
O\left(\frac{1}{k^2}\right)$, $k\to\infty$, and using (\ref{sol}) and (\ref{sol1}), 
the representations (\ref{sol+0}) follow.
\end{proof}

Therefore, using Proposition \ref{propsol}, the large-$t$ asymptotic analysis of $q(x,t)$ reduces 
to  that for a regular RH problem (\ref{RHR}). On the other hand, the latter problem
has the same form as in the case of the NNLS equation on the zero background, see \cite{RS}.
Consequently, one can follows the asymptotic approach, presented  in \cite{RS},
for obtaining the long-time asymptotics for $\hat{M}^R(x,t,k)$ at $k=ik_1$, $k=0$ 
(needed in (\ref{gh})), and for large $k$ (needed in (\ref{sol+0})), which will
finally lead to the long-time asymptotics of $q(x,t)$. 

Before formulating   detailed asymptotics, let us notice that the rough 
approximation \\
$\hat{M}^R(x,t,k)\approx I$ as $t\to\infty$ with $x/t\ge\varepsilon$ for any $\varepsilon>0$
(to avoid the possible singularity of $\delta(\xi,k)$ as $\xi\to 0$), being substituted into 
(\ref{gh}), gives the main term of the asymptotics of $q(x,t)$ with a
rough error estimate:

\begin{proposition}
\label{rough-as}
As $t\to\infty$,
\begin{equation}
\label{asr}
q(x,t)=A\delta^2(\xi,0)+o(1)\ \text{for}\  x>0\quad \text{and}\quad  q(x,t)=o(1)
\ \text{for}\  x<0
\end{equation}
along any ray $\xi=\frac{x}{4t}=const>0$ or   $\xi=const<0$.
\end{proposition}

Indeed, $\hat{M}^R(x,t,k)\approx I$ implies that 
$\left(\begin{smallmatrix}g_1(x,t)\\g_2(x,t) \end{smallmatrix}\right) \approx
 \left(\begin{smallmatrix}ik_1\\ -c_1(x,t) \end{smallmatrix}\right)
\approx
\left(\begin{smallmatrix}ik_1\\ 0 \end{smallmatrix}\right)$
and $\left(\begin{smallmatrix}h_1(x,t)\\h_2(x,t) \end{smallmatrix}\right) \approx
 \left(\begin{smallmatrix}c_0(\xi)\\ ik_1 \end{smallmatrix}\right)$.
Accordingly, 
for $x>0$ we have 
\[
q(x,t) \approx -2k_1P_{12}(x,t) \approx -2k_1\frac{ik_1 c_0(\xi)}{-k_1^2+c_0(\xi)c_1(x,t)}
\approx 2ic_0(\xi)=A\delta^2(\xi,0)
\]
whereas for $x<0$ we have
\[
q(x,t) \approx -2k_1\overline{P_{21}(-x,t)} \approx 2k_1\frac{-\bar c_1(-x,t)(-ik_1)}{-k_1^2+\bar c_0(-\xi) \bar c_1(-x,t)} \approx 0.
\]

Our main results  make (\ref{asr}) more precise. 

\begin{theorem}
\label{th1}
Consider the Cauchy problem (\ref{1}), (\ref{2.0}),
where 
the initial data $q_0(x)$ 
is a compact
perturbation of the pure step initial data (\ref{step-ini}):
$q_0(x)- q_{0A}(x)=0$ for  $|x|>N$ with some $N>0$.
 Assume that the spectral functions
associated with $q_0(x)$ via (\ref{8.1})--(\ref{8.4}) are such that:
\begin{enumerate}[(a)]
\item 
$a_1(k)$ has a single, simple zero in $\overline{\mathbb{C}^{+}}$  at $k=ik_1$, 
 and $a_2(k)$ either has no zeros in $\overline{\mathbb{C}^{-}}$ or has a single, simple zero
at $k=0$.
\item
$\Im \nu(-\xi)\in \left( -\frac{1}{2}, \frac{1}{2} \right)$ for all $\xi>0$, 
where $\Im\nu(-\xi)= -\frac{1}{2\pi}\int_{-\infty}^{-\xi} d \arg(1+ r_1(\zeta)r_2(\zeta))$,
$r_1(k)=\frac{b(k)}{a_1(k)}$, $r_2(k)=\frac{\overline{b(-{k})}}{a_2(k)}$.
\end{enumerate}
Assuming that the solution $q(x,t)$ of (\ref{1}), (\ref{2.0}) exists, its
long-time asymptotics along any line 
$\xi=\frac{x}{4t}=const\ne 0$
is as follows:
\begin{subequations}\label{as-sol}
\begin{description}
\item[(i)] for 
$x<0:$
\begin{equation}
\label{as-sol-1}
q(x,t)=t^{-\frac{1}{2}-\Im\nu(\xi)}\alpha_1(\xi)
\exp\left\{4it\xi^2-i\Re\nu(\xi)\ln t\right\}
+R_1(-\xi,t),
\end{equation}
\item[(ii)] for $x>0$, three types of asymptotics are possible, 
depending on the value of $\Im\nu(-\xi)$:
\begin{description}
\item[(a)] if $\Im\nu(-\xi)\in\left(-\frac{1}{2},-\frac{1}{6}\right]$, then
\begin{equation}
\label{as-sol-2a}
q(x,t)=A\delta^2(\xi,0)+
t^{-\frac{1}{2}-\Im\nu(-\xi)}\alpha_2(\xi)
\exp\left\{-4it\xi^2+i\Re\nu(-\xi)\ln t\right\}
+R_1(\xi,t).
\end{equation}
\item[(b)] if $\Im\nu(-\xi)\in\left(-\frac{1}{6},\frac{1}{6}\right)$, then
\begin{align}
\label{as-sol-2b}
\nonumber
q(x,t)&=A\delta^2(\xi,0)+t^{-\frac{1}{2}+\Im\nu(-\xi)}\alpha_3(\xi)
\exp\left\{4it\xi^2-i\Re\nu(-\xi)\ln t\right\}&\\
&\quad+ 
t^{-\frac{1}{2}-\Im\nu(-\xi)}\alpha_2(\xi)
\exp\left\{-4it\xi^2+i\Re\nu(-\xi)\ln t\right\}
+R_3(\xi,t).
\end{align}
\item[(c)] if $\Im\nu(-\xi)\in\left[\frac{1}{6},\frac{1}{2}\right)$, then
\begin{equation}
\label{as-sol-2c}
q(x,t)=A\delta^2(\xi,0)+t^{-\frac{1}{2}+\Im\nu(-\xi)}\alpha_3(\xi)
\exp\left\{4it\xi^2-i\Re\nu(-\xi)\ln t\right\}+R_2(\xi,t).
\end{equation}
\end{description}
\end{description}
\end{subequations}
Here 
\[
\delta(\xi,0)=\exp\left\{\frac{1}{2\pi i}\int_{-\infty}^{-\xi}\frac{\ln(1+r_1(\zeta)r_2(\zeta))}{\zeta}\,d\zeta\right\},
\]
\[
\nu(-\xi)=-\frac{1}{2\pi}\ln|1+ r_1(-\xi)r_2(-\xi)|-
\frac{i}{2\pi}\Delta(-\xi),
\]
\[
\Delta(-\xi)=\int_{-\infty}^{-\xi}d\arg(1+ r_1(\zeta)r_2(\zeta)),
\]
\[
\alpha_1(\xi)=\begin{cases}
\dfrac{\sqrt{\pi}\,
\exp\left\{-\frac{\pi}{2}\overline{\nu}(\xi)+\frac{\pi i}{4}-2\overline{\chi}(-\xi,\xi)-3i\overline{\nu}(\xi)\ln2\right\}}
{\overline{r_2(\xi)}\Gamma(-i\overline{\nu}(\xi))}, & r_1(-\xi)r_2(-\xi)\neq0, \\
\frac{\overline{r_1(\xi)}e^{\frac{3\pi i}{4}}}{2\sqrt{\pi}}, &  r_1(-\xi)=0,  r_2(-\xi)\ne 0,\\
0, & r_1(-\xi)\ne 0,  r_2(-\xi)= 0,\\
0, &  r_1(-\xi)=r_2(-\xi)=0,
\end{cases}
\]

\[
\alpha_2(\xi)=\begin{cases}
\dfrac{c_0^2(\xi)\sqrt{\pi}\,
	\exp\left\{-\frac{\pi}{2}\nu(-\xi)+\frac{3\pi i}{4}
	-2\chi(\xi,-\xi)+3i\nu(-\xi)\ln2\right\}}
{\xi^2r_2(-\xi)\Gamma(i\nu(-\xi))}, & r_1(-\xi)r_2(-\xi)\neq0, \\
0, &  r_1(-\xi)=0, r_2(-\xi)\neq0, \\
\frac{c_0^2(\xi)r_1(-\xi)e^{\frac{\pi i}{4}}}{2\sqrt{\pi}\xi^2}, &  r_1(-\xi)\ne 0,  r_2(-\xi)= 0,\\
0, &  r_1(-\xi)=r_2(-\xi)=0,
\end{cases}
\]

\[
\alpha_3(\xi)=\begin{cases}
\dfrac{\sqrt{\pi}\,
\exp\left\{-\frac{\pi}{2}\nu(-\xi)+\frac{\pi i}{4}+2\chi(\xi,-\xi)-3i\nu(-\xi)\ln2\right\}}
{r_1(-\xi)\Gamma(-i\nu(-\xi))},  & r_1(-\xi)r_2(-\xi)\neq0, \\
\frac{r_2(-\xi)e^{\frac{3\pi i}{4}}}{2\sqrt{\pi}}, 
&  r_1(-\xi)=0, r_2(-\xi)\neq0, \\
0, & r_1(-\xi)\ne 0,  r_2(-\xi)= 0,\\
0, &  r_1(-\xi)=r_2(-\xi)=0,
\end{cases}
\]
with
\[
{\chi(\xi,k)}=-\frac{1}{2\pi i}\int_{-\infty}^{-\xi}\ln(k-\zeta)d_{\zeta}\ln(1+ r_1(\zeta)r_2(\zeta)),
\]
where $\Gamma(\cdot)$ is the Euler Gamma-function.

The error estimates
 $R_1(\xi,t)$ and $R_2(\xi,t)$
 are uniform in any compact subset of $\xi\in (0,\infty)$ and are as follows:
\begin{equation}
\label{R1}
R_1(\xi,t)=
\begin{cases}
O\left(t^{-1}\right),& \Im\nu(-\xi)>0,\\
O\left(t^{-1}\ln t\right),&\Im\nu(-\xi)=0,\\
O\left(t^{-1+2|\Im\nu(-\xi)|}\right),&\Im\nu(-\xi)<0,
\end{cases}
\end{equation}
\begin{equation}
\label{R2}
R_2(\xi,t)=
\begin{cases}
O\left(t^{-1+2|\Im\nu(-\xi)|}\right),& \Im\nu(-\xi)>0,\\
O\left(t^{-1}\ln t\right),&\Im\nu(-\xi)=0,\\
O\left(t^{-1}\right),&\Im\nu(-\xi)<0,
\end{cases}
\end{equation}
and
\begin{equation*}
R_3(\xi,t)=R_1(\xi,t) + R_2(\xi,t)=
\begin{cases}
O\left(t^{-1+2|\Im\nu(-\xi)|}\right),&\Im\nu(-\xi)\not=0,\\
O\left(t^{-1}\ln t\right),&\Im\nu(-\xi)=0.
\end{cases}
\end{equation*}
\end{theorem}

\begin{remark}
Notice  that $\delta(\xi,0)\rightarrow1$ as $\xi\rightarrow\infty$ and thus 
the asymptotics
(\ref{as-sol-2a})-(\ref{as-sol-2c}) is consistent with 
the boundary conditions (\ref{2-b}).
\end{remark}

\begin{remark}
In the case of the pure-step initial data, i.e. $q(x,0)=0$ for $x<0$ and $q(x,0)=A$ for $x\geq0$, both assumptions of the theorem hold true. Moreover, in this case
$1+r_1(k)r_2(k)=\frac{4k^2}{4k^2+A^2}$ and thus $\Im\nu=0$ in (\ref{as-sol}).
\end{remark}

\begin{remark}
The problem of describing asymptotic transition between the regions
$x<0$ and $x>0$ remains open. Some observations showing that this problem is
far nontrivial   are as follows:
\begin{enumerate}
\item
From the point of view  of the Riemann-Hilbert problem formalism, 
the transition region corresponds to merging the stationary phase point
$k=-\xi$ and the singularity point $k=0$; to the best of our knowledge,
such transition picture has not been considered in the literature.
\item
The main asymptotic term for $x>0$, $A\delta^2(\xi, 0)$, 
develops, in general, increasing oscillations as $\xi\to+0$;
only in very particular cases (belonging to Case II only), 
where $b(0)=0$, there exist a finite limit of $\delta(\xi, 0)$
as $\xi\to+0$.
\item
Even in the simplest case of a soliton solution,
where the asymptotics holds as $|x|$ increases, together with $t$,
along any path in the half-planes $x>0$ and $x<0$,
(in this case we have $\nu\equiv 0$ and thus $\delta(\xi,k)\equiv 1)$,
at the boundary line $x=0$ the solution develops discrete (in $t$) singularities.
\end{enumerate}
\end{remark}

\begin{remark}
In the case of  the initial data $q_0(x)$ such that $b(0)=0$, the asymptotics of the solution 
 for a fixed $x=x_0\in\mathbb{R}$ (which corresponds to $\xi\to 0$) has the form:
\begin{equation}\label{as-soliton}
q(x_0,t)=\frac{2iAk_1^2\dot a_1(ik_1)\delta^2(0,ik_1)\exp\{2\hat\chi_1\}}
{2ik_1^2\dot a_1(ik_1)\delta^2(0,ik_1)-A\gamma_1
\exp\{-2k_1x_0-4ik_1^2t+2\hat\chi_1\}}+o(1), \quad t\to\infty
\end{equation}
with
\begin{equation}\label{chi1}
\hat\chi_1=\frac{i}{2\pi}\int_{-\infty}^{0}\ln(-\zeta)\,
d_{\zeta}\ln(1+r_1(\zeta)r_2(\zeta)),
\end{equation}
where  $x_0$ and $t$ are such that the denominator 
in (\ref{as-soliton}) 
is not equal to zero, i.e.
\begin{equation}\label{condzero}
\exp\{-2k_1x_0-4ik_1^2t\}\not=\frac{2ik_1^2}{A\gamma_1}\dot a_1(ik_1)
\exp\left\{\frac{k_1}{\pi}\int_{-\infty}^{0}\frac{\ln(1+r_1(\zeta)r_2(\zeta))}{\zeta(\zeta-ik_1)}\,d\zeta\right\}.
\end{equation}
Indeed,  the solution $\hat M^{R}$ of the regular RH problem (\ref{RHR}) has the following asymptotics for all $\xi\geq0$:
\begin{equation}\label{asM}
\hat M^{R}(x,t,k)=I+o(1),\quad t\to\infty,\quad x\geq 0.
\end{equation}
Integrating (\ref{14}) by parts in Case II (notice that $b(0)=0$ belongs to Case II),
we have that 
\begin{equation}
\delta(\xi,0)\sim\exp\left\{\frac{i}{2\pi}\ln\xi\cdot\ln a_{11}\dot a_2(0)+\hat\chi_1\right\}
\quad\text{as}\ \xi\to +0.
\end{equation}
Moreover, if $b(0)=0$, then 
 $a_{11}\dot a_2(0)=1-|b(0)|^2=1$ and thus $\delta(0,0)=e^{\hat\chi_1}$ which implies that $c_0(\xi)$ (see (\ref{c0})) is well-defined for $\xi=0$.
Consequently,  (\ref{sol+0a}) is valid for all $x\geq0$ and $t>0$ such that $P_{12}(x,t)$ and $P_{21}(x,t)$ (see (\ref{P}) and (\ref{gh})) have nonzero denominators. Evaluating $P_{12}(x,t)$ and $P_{21}(x,t)$ in (\ref{sol+0})
and using (\ref{asM}), 
we conclude that as $t\to\infty$, 
\begin{subequations}
\begin{align}
\label{soliton+0}
&q(x_0,t)=\frac{2iAk_1^2\dot a_1(ik_1)\delta^2(0,ik_1)\exp\{2\hat\chi_1\}}
{2ik_1^2\dot a_1(ik_1)\delta^2(0,ik_1)-A\gamma_1
	\exp\{-2k_1x_0-4ik_1^2t+2\hat\chi_1\}}+o(1),
\quad \text{if}\ x_0\geq 0,\\
\label{soliton-0}
&q(x_0,t)=
\frac{4k_1^2}{2ik_1^2\gamma_1\overline{\dot a_1(ik_1)}\overline{\delta^2(0,ik_1)}
\exp\{-2k_1x_0-4ik_1^2t\}+A\exp\{2\overline{\hat\chi_1}\}}+o(1),
\quad \text{if}\ x_0< 0.
\end{align}
\end{subequations}
Taking into account (\ref{k1-ngen}), (\ref{tracenga}), (\ref{14}) and 
using the equality  $1+r_1(k)r_2(k)=(1-b(k)\overline{b(-k)})^{-1}$, we have
\[
k_1\overline{\dot a_1(ik_1)}\overline{\delta^2(0,ik_1)}=
\frac{1}{k_1\dot a_1(ik_1)\delta^2(0,ik_1)}\quad 
\text{and} \quad 
\frac{A\exp\{2\overline{\hat\chi_1}\}}{2k_1}=
\frac{2k_1}{A\exp\{2\hat\chi_1\}},
\]
which implies that the formula for the principal term in (\ref{soliton-0}) 
coincides with  that in
(\ref{soliton+0}) and thus we arrive at
 (\ref{as-soliton}) for all $x_0\in\mathbb{R}$.

Finally, we notice that in the reflectionless case ($b(k)\equiv 0$),
we have $\delta\equiv 1$, $\hat\chi_1=0$, $k_1=\frac{A}{2}$, 
$\dot a_1(ik_1)=-\frac{2i}{A}$  and thus 
the principal term in  (\ref{as-soliton}) reduces 
 to the pure soliton solution (\ref{pure-sol}).
\end{remark}
\noindent\textit{Sketch of proof of Theorem \ref{th1}}.

Here we consider the case $r_j(-\xi)\neq0$, $j=1,2$ (for  the cases when one of the $r_j(-\xi)$ (or the both) equals  zero and thus $\nu(-\xi)=0$, we refer to   Section 1.5 of Chapter 2 in \cite{FIKN}). In view of (\ref{sol+0}), for obtaining the asymptotics (\ref{as-sol}) it is sufficient to estimate the solution $\hat{M}^{R}(x,t,k)$ of the regular RH problem (\ref{RHR}) at  
$k=0$, $k=ik_1$ and $k=\infty$. Noticing that this RH problem 
is similar to that in the case of decaying initial data \cite{RS},
in what follows we will refer to \cite{RS} for the details of the relative steps in the asymptotic analysis.

First, introduce the rescaled variable $z$ 
by
\begin{equation}\label{k-z}
 k=\frac{z}{\sqrt{8t}}-\xi,
\end{equation}
so that 
\[
e^{2it\theta} = e^{\frac{iz^2}{2}-4it\xi^2}.
\]

Introduce the ``local parametrix'' $\hat m_0^R(x,t,k)$
as the solution of a RH problem with the jump matrix that is a ``simplified $\hat{J}^{R}(x,t,k)$''
in the sense that in its construction, $r_j(k)$, $j=1,2$ are replaced by the constants $r_j(-\xi)$
and $\delta(\xi,k)$ is replaced by (cf. (\ref{delta-singular}))
$\delta\simeq\left(\frac{z}{\sqrt{8t}}\right)^{i\nu(-\xi)}e^{\chi(\xi,-\xi)}$. 
Such RH problem can be solved explicitly, 
in terms of the parabolic cylinder functions \cite{I1,RS}.

Indeed, $\hat m_0^R(x,t,k)$ can be determined by 
\begin{equation}\label{m0-R}
\hat m_0^R(x,t,k)=\Delta(\xi,t)m^{\Gamma}(\xi,z(k)) \Delta^{-1}(\xi,t),
\end{equation}
where 
\begin{equation}\label{Delta}
\Delta(\xi,t) = e^{(2 i t \xi^2 + \chi(\xi,-\xi))\sigma_3}(8t)^{-\frac{i\nu(-\xi)}{2}\sigma_3},
\end{equation}
$m^{\Gamma}(\xi,z)$ is determined by 
\begin{equation}\label{m-g-0}
m^{\Gamma}(\xi,z) = m_0(\xi,z) D^{-1}_{j}(\xi,z),\qquad z\in\Omega_j,\,\,j=0,\ldots,4,
\end{equation}
see Figure \ref{mod2},
where $\gamma_j$ corresponds to $\hat\gamma_j$ in accordance with (\ref{k-z}).
Here  $D_0(\xi,z)=e^{-i\frac{z^2}{4}\sigma_3}z^{i\nu(-\xi)\sigma_3}$, 
\[
\begin{aligned}
 & D_1(\xi,z)=D_0(\xi,z)
\begin{pmatrix}
1& \frac{r_2^R(-\xi)}{1+ r_1^R(-\xi)r_2^R(-\xi)}\\
0& 1\\
\end{pmatrix},
& \qquad & 
D_2(\xi,z)=D_0(\xi,z)
\begin{pmatrix}
1& 0\\
r_1^R(-\xi)& 1\\
\end{pmatrix}, \\
 & D_3(\xi,z)=D_0(\xi,z)
\begin{pmatrix}
1& - r_2^R(-\xi)\\
0& 1\\
\end{pmatrix},
& &
D_4(\xi,z)=D_0(\xi,z)
\begin{pmatrix}
1& 0\\
\frac{-r_1^R(-\xi)}{1+ r_1^R(-\xi)r_2^R(-\xi)}& 1
\end{pmatrix}
\end{aligned}
\]
with
$$
r_1^R(k)= \frac{k-ik_1}{k}r_1(k),\quad r_2^{R}(k)=\frac{k}{k-ik_1}r_2(k),
$$
and $m_0(\xi,z)$ is the solution of the following RH problem in $z$-plane 
(relative to $\mathbb R$,
 with a \emph{constant jump matrix}):
\begin{equation}\label{as8}
\begin{cases}
m_{0+}(\xi,z)=m_{0-}(\xi,z)j_0(\xi),& z\in\mathbb{R},\\
m_0(\xi,z)= \left(I+O(1/z)\right)
e^{-i\frac{z^2}{4}\sigma_3}z^{i\nu(-\xi)\sigma_3},& z\rightarrow\infty,
\end{cases}
\end{equation}
where  
\begin{equation}\label{j0}
j_0(\xi)=
\begin{pmatrix}
1+ r_1^R(-\xi)r_2^R(-\xi) & r_2^R(-\xi)\\
r_1^R(-\xi) & 1
\end{pmatrix}.
\end{equation}

\begin{figure}[ht]
\centering{\includegraphics[scale=0.1]{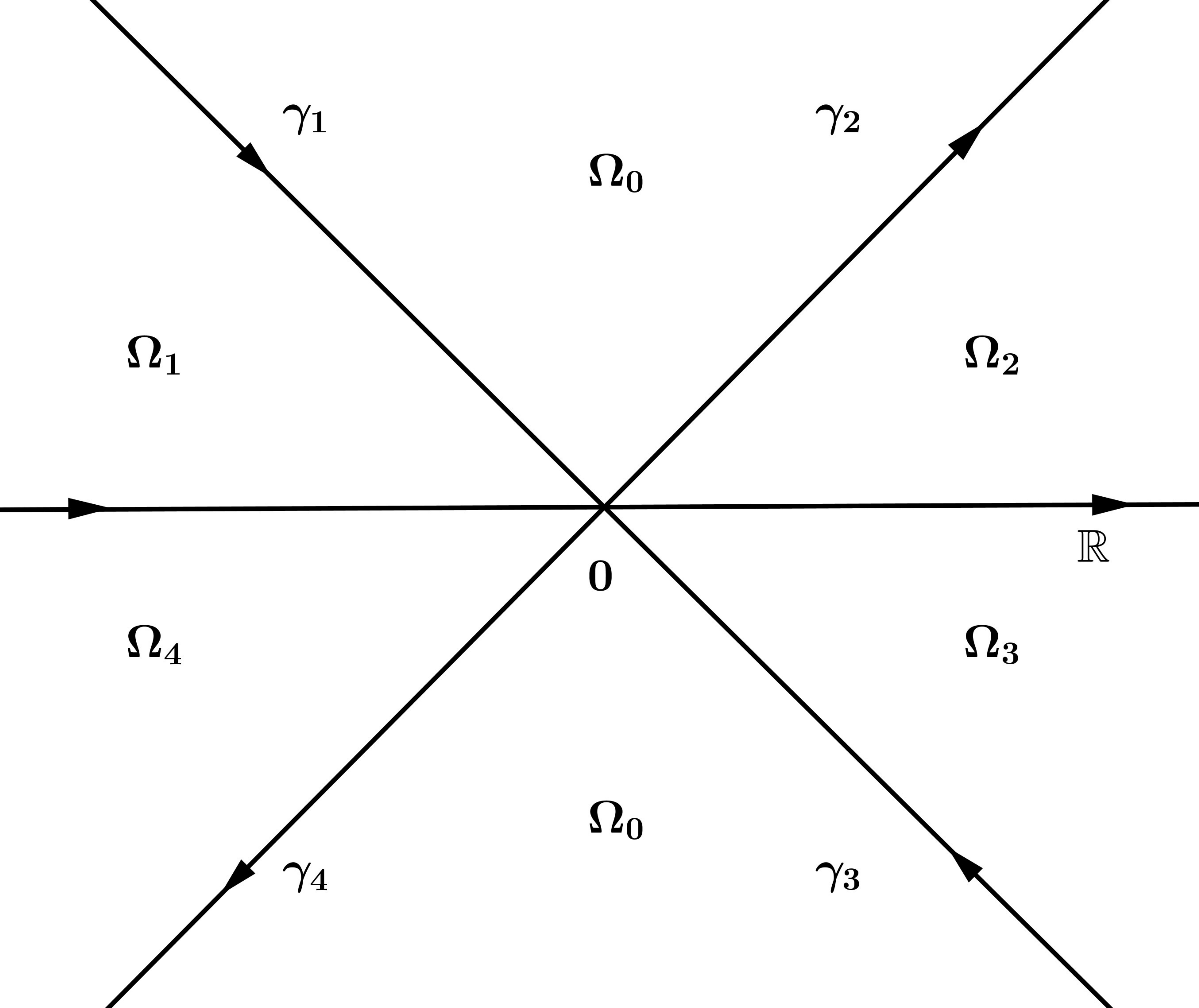}}
\caption{Contour and domains for $m^{\Gamma}(\xi,z)$ in the $z$-plane  }
\label{mod2}
\end{figure}

It is the RH problem for $m_0(\xi,z)$ that can be solved explicitly, in terms of the parabolic
cylinder functions, see, e.g., Appendix A in \cite{RS}.
Since we are interested in what happens for large $t$ and, in view of (\ref{k-z}), even finite values of $k$ correspond to large
values of $z$ if $t$ is large, it follows that all we actually need from $m_0(\xi,z)$ (and, correspondingly, 
$m^{\Gamma}(\xi,z)$) is its large-$z$ asymptotics only. The latter has  the form
\[
m^{\Gamma}(\xi,z) = I + \frac{i}{z}\begin{pmatrix}
	0 & \beta^R(\xi) \\ -\gamma^R(\xi) & 0
\end{pmatrix} + O(z^{-2}), \qquad z\to \infty,
\]
where (cf.  $\beta(\xi)$ and $\gamma(\xi)$ in \cite{RS}) 
\begin{subequations}\label{be-ga-R}
\begin{align}
\beta^R(\xi)=\dfrac{\sqrt{2\pi}e^{-\frac{\pi}{2}\nu(-\xi)}e^{-\frac{3\pi i}{4}}}{r_1^R(-\xi)
\Gamma(-i\nu(-\xi))},\\
\gamma^R(\xi)=\dfrac{\sqrt{2\pi}e^{-\frac{\pi}{2}\nu(-\xi)}e^{-\frac{\pi i}{4}}}{r_2^R(-\xi)\Gamma(i\nu(-\xi))}.
\end{align}
\end{subequations}

Now, having defined the parametrix $\hat m_0^R(x,t,k)$,
we define $\check M^{R}(x,t,k)$ (cf.  $\hat{m}(x,t,k)$ in \cite{RS}) as follows:
\[
\check M^R(x,t,k) = 
\begin{cases}
\hat M^R(x,t,k)(\hat m_0^{R})^{-1}(x,t,k), & |k+\xi|<\varepsilon, \\
\hat M^R(x,t,k), & \mbox{otherwise},
\end{cases}
\]
where $\varepsilon$ is small enough so that $|\xi|>\varepsilon$ and $|ik_1+\xi|>\varepsilon$. Then the sectionally analytic matrix $\check M^{R}$ has the following jumps across  
$\hat\Gamma_1=\hat\Gamma\cup\{|k+\xi|=\varepsilon\} $ (the circle 
$|k+\xi|=\varepsilon $ is  oriented counterclockwise)
\begin{equation}\label{check-J}
\check J^R(x,t,k) = 
\begin{cases}
\hat m_{0-}^R(x,t,k) \hat J^R(x,t,k)(\hat m_{0+}^R)^{-1} (x,t,k)
, & k\in \hat\Gamma, |k+\xi|<\varepsilon, \\
\left(\hat m_{0}^{R}\right)^{-1} (x,t,k), & |k+\xi|=\varepsilon, \\
\hat J^R(x,t,k), & \text{otherwise}.
\end{cases}
\end{equation}

The next step is the large-$t$ evaluation of  $\check M^{R}(x,t,k)$ using its
 representation  in terms of the solution of the  singular integral equation
corresponding to the RH problem determined by the jump conditions 
(\ref{check-J}) and the standard normalization condition $\check M^{R}\to I$ as $k\to\infty$.
We have
\begin{equation}\label{M-int-rep}
\check M^{R}(x,t,k) = I+\frac{1}{2\pi i}\int_{\hat\Gamma_1}\mu(x,t,s)(\check J^{R}(x,t,s)-I)
\frac{ds}{s-k},
\end{equation}
where $\mu$ solves the integral equation $\mu -C_w \mu = I$, with 
$w=\check J^{R} - I$.
Here the Cauchy-type operator $C_w$ is defined by 
$C_w f = C_-(fw)$, where 
$(C_-h)(k)$, $k\in \hat\Gamma_1$ are the right (according to the orientation of 
$\hat\Gamma_1$) non-tangential boundary values of 
\[
(Ch)(k')=\frac{1}{2\pi i}\int_{\hat\Gamma_1}\frac{h(s)}{s-k'}ds, \quad k'\in{\mathbb C}\setminus
\hat\Gamma_1.
\]

Reasoning as in \cite{RS} one can show that the main term in the large-$t$ development 
of $\check M^{R}$ in (\ref{M-int-rep}) is given by the integral along the circle
$|s+\xi|=\varepsilon$, which in turn gives
\begin{equation}\label{M^R}
\check M^R(x,t,k)=I-\frac{1}{2\pi i}\int_{|s+\xi|=\varepsilon}
\frac{\tilde{B}^{R}(\xi,t)}{(s+\xi)(s-k)}\,ds+R(\xi,t),
\quad |k+\xi|>\varepsilon,
\end{equation}
where
\begin{equation}\label{tilde-B}
\tilde{B}^R(\xi,t)=\begin{pmatrix}
0 & i\beta^R(\xi)e^{4it \xi^2 + 2\chi(\xi,-\xi)}(8t)^{-\frac{1}{2}-i\nu(-\xi)} \\
-i\gamma^R(\xi)e^{-4it \xi^2 - 2\chi(\xi,-\xi)}(8t)^{-\frac{1}{2}+i\nu(-\xi)} & 0
\end{pmatrix}
\end{equation}
and
the (matrix) error estimate $R$ has the structure 
$R(\xi,t)=\left(\begin{smallmatrix}
R_1(\xi,t) & R_2(\xi,t)\\
R_1(\xi,t) & R_2(\xi,t)
\end{smallmatrix}\right)$, with $R_1$ and $R_2$ having, in general,
different orders of decay, see (\ref{R1}) and (\ref{R2}).
Particularly, since $\check M^{R}=\hat M^{R}$ for all $k$ with
 $|k+\xi|>\varepsilon$, we have 
\begin{equation}\label{kinfty}
\lim\limits_{k\to\infty}k\left(\hat M^{R}(x,t,k)-I\right)=
\tilde{B}^{R}(\xi,t)+R(\xi,t)
\end{equation}
as well as
\begin{subequations}\label{MR0}
\begin{align}
\hat M^R(x,t,0)=I+\frac{\tilde B^{R}(\xi,t)}{\xi}+R(\xi,t),\\
\label{MRik_1}
\hat M^R(x,t,ik_1)=I+\frac{\tilde B^{R}(\xi,t)}{\xi+ik_1}+R(\xi,t).
\end{align}
\end{subequations}

Now we are at a position to evaluate 
$P_{12}(x,t)$ and $P_{21}(x,t)$ in (\ref{sol+0}).
First, we evaluate $g_j(x,t)$ and $h_j(x,t)$, $j=1,2$,  defined in (\ref{gh}),
using (\ref{MR0}) and replacing   $\hat M^{R}$ by $\check M^{R}$:
\begin{align}
\nonumber
g_1(x,t)&=ik_1+R_1(\xi,t),
&&g_2(x,t)=\frac{ik_1}{\xi+ik_1}\tilde{B}^{R}_{21}(\xi,t)+R_1(\xi,t),\\
\nonumber
h_1(x,t)&=c_0(\xi)+\frac{ik_1}{\xi}\tilde{B}^{R}_{12}(\xi,t)
+R_3(\xi,t),
&&h_2(x,t)=ik_1+\frac{c_0(\xi)}{\xi}\tilde{B}^{R}_{21}(\xi,t)
+R_3(\xi,t), 
\end{align}
where $R_3(\xi,t)=R_1(\xi,t)+R_2(\xi,t)$ 
 (we have used the standard notation for the entries of matrix $\tilde{B}^{R}(\xi,t)$). 
It follows that (we drop the arguments of the functions)
\begin{subequations}\label{gh2}
\begin{align}
g_1h_1&=ik_1c_0(\xi)-\frac{k_1^2}{\xi}\tilde{B}^{R}_{12}+R_3,
&&g_1h_2=-k_1^2+\frac{ik_1c_0(\xi)}{\xi}\tilde{B}^{R}_{21}+R_3,\\
g_2h_1&=\frac{ik_1c_0(\xi)}{\xi+ik_1}\tilde{B}^{R}_{21}+R_1,
&&g_2h_2=-\frac{k_1^2}{\xi+ik_1}\tilde{B}^{R}_{21}+R_1.
\end{align}
\end{subequations}
Substituting (\ref{gh2}) into (\ref{P}),  straightforward calculations give
\begin{subequations}\label{P-as}
\begin{align}
P_{12}(x,t)&=-\frac{ic_0(\xi)}{k_1}+\frac{\tilde{B}^R_{12}(\xi,t)}{\xi}+
\frac{ic_0(\xi)^2}
{\xi k_1(\xi+ik_1)}\tilde{B}^R_{21}(\xi,t)+R_3(\xi,t),\\
P_{21}(x,t)&=-\frac{\tilde{B}^R_{21}(\xi,t)}{\xi+ik_1}+R_1(\xi,t).
\end{align}
\end{subequations}

Notice that formulas (\ref{P-as}) involve $k_1$ explicitly. But
using 
\[
\tilde{B}^R_{12} = \tilde{B}_{12}\frac{\xi}{\xi+ik_1}, \qquad 
\tilde{B}^R_{21} = \tilde{B}_{21}\frac{\xi+ik_1}{\xi},
\]
where $\tilde{B}$ is defined similarly to $\tilde{B}^R$, 
see (\ref{be-ga-R}) and (\ref{tilde-B}),
with $r_j^R(-\xi)$ replaced by $r_j(-\xi)$, and 
substituting (\ref{kinfty}) and (\ref{P-as}) into (\ref{sol+0}),
it follows that
 the (explicit) dependence on $k_1$ in the resulting formulas for the main asymptotic terms 
vanishes, and we arrive at
the asymptotic formulas (\ref{as-sol}).


\end{document}